\documentclass[reqno]{amsart}
\usepackage{amssymb,latexsym}
\usepackage{graphicx}
\usepackage{fancyhdr}
\numberwithin{equation}{section}
\newtheorem{thm}{Theorem}[section]
\newtheorem{theorem}[thm]{Theorem}

\newtheorem{lemma}[thm]{Lemma}

\newtheorem{corollary}[thm]{Corollary}

\begin{document}

\setcounter{page}{1}

\title[Reciprocity formulas]{Some reciprocity formulas for generalized Dedekind-Rademacher sums}
\author{Yuan He}
\address{Data Recovery Key Laboratory of Sichuan Province, College of Mathematics and Information Science, Neijiang Normal University, Neijiang 641100, Sichuan, People's Republic of China}
\email{hyyhe@njtc.edu.cn, hyyhe@aliyun.com}
\author{Yong-Guo Shi}
\address{Data Recovery Key Laboratory of Sichuan Province, College of Mathematics and Information Science, Neijiang Normal University, Neijiang 641100, Sichuan, People's Republic of China}
\email{scumat@163.com}

\begin{abstract}
In this paper, we study the generalized Dedekind-Rademacher sums considered by Hall, Wilson and Zagier. We establish a formula for the products of two Bernoulli functions. The proof relies on Parseval's formula, Hurwitz's formula, and Lerch's functional equation. The result leads to reciprocity formulas for some generalizations of the classical Dedekind sums. In particular, it is shown that Carlitz's, Berndt's, Hall and Wilson's reciprocity theorems are deduced as special cases.
\end{abstract}

\maketitle

\section{Introduction}

Let $\mathbb{N}$ be the set of positive integers, $\mathbb{N}_{0}$ the set of non-negative integers, $\mathbb{Z}$ the set of integers, $\mathbb{R}$ the set of real numbers, and $\mathbb{C}$ the set of complex numbers. Let $\{x\}$ be the fractional part of $x\in\mathbb{R}$, and denote by $((x))$ the sawtooth function (also called the first Bernoulli function) given for $x\in\mathbb{R}$ by
\begin{equation*}
((x))=\begin{cases}
\{x\}-\frac{1}{2},  &\text{if $x\in\mathbb{R}\setminus\mathbb{Z}$},\\
0,  &\text{if $x\in\mathbb{Z}$}.
\end{cases}
\end{equation*}
While studying the transformation formula of the Dedekind eta function, Dedekind \cite{dedekind} in the 1880's naturally arrived at the classical Dedekind sums $s(a,b)$ defined for $a,b\in\mathbb{Z}$ with $b\not=0$ by
\begin{equation}\label{eq1.1}
s(a,b)=\sum_{r=1}^{|b|}\biggl(\biggl(\frac{r}{b}\biggl)\biggl)\biggl(\biggl(\frac{ar}{b}\biggl)\biggl),
\end{equation}
and used the transformation formula to discover his reciprocity theorem: If $a,b$ are two relatively prime positive integers then
\begin{equation}\label{eq1.2}
s(a,b)+s(b,a)=-\frac{1}{4}+\frac{1}{12}\biggl(\frac{a}{b}+\frac{1}{ab}+\frac{b}{a}\biggl).
\end{equation}
Interest in the classical Dedekind sums and their generalizations has since been aroused by numerous mathematicians from different areas such as number theory, algebraic, geometry, topology and so on (see, for example, \cite{beck1,beck2,gunnells,rademacher3,yau,zagier,zhang}). One of the most intriguing and important features for these sums is the reciprocity formula. In the present paper, we will be concerned with the generalizations for Dedekind's reciprocity theorem in the direction of the Bernoulli functions. It is well known that Rademacher \cite{rademacher1} in 1954 generalized \eqref{eq1.2}, and proved his three-term reciprocity theorem: If $a,b,c$ are three relatively prime positive integers, and if $a',b',c'\in\mathbb{Z}$ satisfy $aa'\equiv 1$ (mod $bc$), $bb'\equiv 1$ (mod $ca$), $cc'\equiv 1$ (mod $ab$), then
\begin{equation}\label{eq1.3}
s(bc',a)+s(ca',b)+s(ab',c)=-\frac{1}{4}+\frac{1}{12}\biggl(\frac{a}{bc}+\frac{b}{ca}+\frac{c}{ab}\biggl).
\end{equation}
Dieter \cite{dieter} in 1957 improved the restrictions on $a',b',c'$, and gave a proof of \eqref{eq1.3} in the case when $a',b',c'\in\mathbb{Z}$ such that $aa'\equiv 1$ (mod $b$), $bb'\equiv 1$ (mod $c$), $cc'\equiv 1$ (mod $a$). Another way of generalizing \eqref{eq1.2} was also invented by Rademacher \cite{rademacher2}, who, in 1964, introduced the so-called Dedekind-Rademacher sums $s(a,b;x,y)$ given for $a,b\in\mathbb{Z},x,y\in\mathbb{R}$ with $b\not=0$ by
\begin{equation}\label{eq1.4}
s(a,b;x,y)=\sum_{r=1}^{|b|}\biggl(\biggl(\frac{r+y}{b}\biggl)\biggl)\biggl(\biggl(\frac{a(r+y)}{b}+x\biggl)\biggl),
\end{equation}
and shown that if $a,b$ are two relatively prime positive integers then for $x,y\in\mathbb{R}$,
\begin{eqnarray}\label{eq1.5}
&&s(a,b;x,y)+s(b,a;y,x)\nonumber\\
&&=-\frac{\delta_{\mathbb{Z}}(x)\delta_{\mathbb{Z}}(y)}{4}+((x))((y))+\frac{a\overline{B}_{2}(y)}{2b}+\frac{b\overline{B}_{2}(x)}{2a}+\frac{\overline{B}_{2}(ay+bx)}{2ab},
\end{eqnarray}
where $\delta_{\mathbb{Z}}(x)=1$ or $0$ according to $x\in\mathbb{Z}$ or $x\not\in\mathbb{Z}$, $\overline{B}_{2}(x)$ is the second Bernoulli function given for $x\in\mathbb{R}$ by $\overline{B}_{2}(x)=B_{2}(\{x\})$ with $B_{2}(x)=x^{2}-x+\frac{1}{6}$ being the Bernoulli polynomials of degree $2$.
In the year 1977, Berndt \cite{berndt3} systematically studied the generalizations of the classical Dedekind sums that involve the first Bernoulli function, and demonstrated reciprocity theorems for various types of Dedekind sums as well as providing some new proofs of Rademacher's reciprocity formulas \eqref{eq1.3} and \eqref{eq1.5}. In particular, Berndt \cite[Theorem 4.3]{berndt3} explored the sums $s(a,b,c;x,y,z)$ given for $a,b,c\in\mathbb{Z},x,y,z\in\mathbb{R}$ with $c\not=0$ by
\begin{equation}\label{eq1.6}
s(a,b,c;x,y,z)=\sum_{r=1}^{|c|}\biggl(\biggl(\frac{a(r+z)}{c}-x\biggl)\biggl)\biggl(\biggl(\frac{b(r+z)}{c}-y\biggl)\biggl),
\end{equation}
and discovered that for $a,b,c\in\mathbb{N},x,y,z\in\mathbb{R}$,
\begin{eqnarray}\label{eq1.7}
&&s(a,b,c;x,y,z)+s(b,c,a;y,z,x)+s(c,a,b;z,x,y)\nonumber\\
&&=-\frac{N}{4}+\frac{c(a,b)^{2}\overline{B}_{2}\bigl(\frac{ay-bx}{(a,b)}\bigl)}{2ab}+\frac{a(b,c)^{2}\overline{B}_{2}\bigl(\frac{bz-cy}{(b,c)}\bigl)}{2bc}\nonumber\\
&&\quad+\frac{b(c,a)^{2}\overline{B}_{2}\bigl(\frac{cx-az}{(c,a)}\bigl)}{2ca},
\end{eqnarray}
where $(a,b)$ denotes the greatest common factor of $a,b\in\mathbb{Z}$, $N$ is the number of distinct triples $r,s,t\in\mathbb{Z}$ such that
\begin{equation*}
0\leq\frac{r+x}{a}=\frac{s+y}{b}=\frac{t+z}{c}<1.
\end{equation*}
Obviously, Berndt's reciprocity formula \eqref{eq1.7} extends and unifies Rademacher's reciprocity formulas \eqref{eq1.3} and \eqref{eq1.5}. We here mention that Berndt and Evans \cite{berndt4} have used \eqref{eq1.5} to prove some formulas for the class number of an imaginary quadratic number field.

We now turn to another generalizations of the classical Dedekind sums that involve the higher order Bernoulli functions. Let $B_{n}(x)$ be the Bernoulli polynomial of degree $n$, and $\overline{B}_{n}(x)$ the $n$-th Bernoulli function given for $n\in\mathbb{N}_{0}$, $x\in\mathbb{R}$ by
\begin{equation}\label{eq1.8}
\overline{B}_{0}(x)=1, \quad\overline{B}_{1}(x)=((x)),\quad\overline{B}_{n}(x)=B_{n}(\{x\})\quad(n\geq2).
\end{equation}
In the year 1950, Apostol \cite{apostol1} introduced the generalized Dedekind sums $s_{n}(a,b)$ defined for $n\in\mathbb{N}$, $a,b\in\mathbb{Z}$ with $b\not=0$ by
\begin{equation}\label{eq1.9}
s_{n}(a,b)=\sum_{r=1}^{|b|}\overline{B}_{1}\biggl(\frac{r}{b}\biggl)\overline{B}_{n}\biggl(\frac{ar}{b}\biggl),
\end{equation}
and obtained his reciprocity theorem: If $a,b$ are two relatively prime positive integers then for an odd positive integer $n$,
\begin{eqnarray}\label{eq1.10}
&&ab^{n}s_{n}(a,b)+ba^{n}s_{n}(b,a)\nonumber\\
&&=\frac{1}{n+1}\sum_{j=0}^{n+1}\binom{n+1}{j}(-1)^{j}a^{j}b^{n+1-j}B_{j}B_{n+1-j}+\frac{nB_{n+1}}{n+1},
\end{eqnarray}
where $B_{j}$ is the $j$-th Bernoulli number given by $B_{j}=B_{j}(0)$. In order to extend the reciprocity formulas \eqref{eq1.5} and \eqref{eq1.10}, Carlitz \cite{carlitz2} in 1964 considered the sums $\widetilde{s}_{n}(a,b;x,y)$ defined for $n\in\mathbb{N}_{0}$, $a,b\in\mathbb{Z}$, $x,y\in\mathbb{R}$ with $b\not=0$ by
\begin{equation}\label{eq1.11}
\widetilde{s}_{n}(a,b;x,y)=\sum_{r=1}^{|b|}B_{1}\biggl(\biggl\{\frac{r+y}{b}\biggl\}\biggl)B_{n}\biggl(\biggl\{\frac{a(r+y)}{b}+x\biggl\}\biggl),
\end{equation}
and found that if $a,b$ are two relatively prime positive integers then for $n\in\mathbb{N}_{0}$, $x,y\in\mathbb{R}$,
\begin{eqnarray}\label{eq1.12}
&&ab^{n}\widetilde{s}_{n}(a,b;x,y)+ba^{n}\widetilde{s}_{n}(b,a;y,x)\nonumber\\
&&=\frac{1}{n+1}\sum_{j=0}^{n+1}\binom{n+1}{j}a^{j}b^{n+1-j}B_{j}(\{y\})B_{n+1-j}(\{x\})\nonumber\\
&&\quad+\frac{nB_{n+1}(\{ay+bx\})}{n+1}.
\end{eqnarray}
The sums \eqref{eq1.11} were further extended by Hall, Wilson and Zagier \cite{hall2} to the following generalized Dedekind-Rademacher sums
\begin{equation}\label{eq1.13}
s_{m,n}\left(
\begin{matrix}
a & b & c \\
x & y & z
\end{matrix}
\right)=\sum_{r=1}^{|c|}\overline{B}_{m}\biggl(\frac{a(r+z)}{c}-x\biggl)\overline{B}_{n}\biggl(\frac{b(r+z)}{c}-y\biggl),
\end{equation}
where $m,n\in\mathbb{N}_{0}$, $a,b,c\in\mathbb{Z}$, $x,y,z\in\mathbb{R}$ with $c\not=0$. In particular, Hall, Wilson and Zagier \cite{hall2} proved that the generating function of the sums \eqref{eq1.13} satisfies the following three-term reciprocity formula
\begin{eqnarray}\label{eq1.14}
&&\Omega\left(
\begin{matrix}
a & b & c \\
x & y & z \\
X & Y & Z
\end{matrix}
\right)+\Omega\left(
\begin{matrix}
b & c & a \\
y & z & x \\
Y & Z & X
\end{matrix}
\right)+\Omega\left(
\begin{matrix}
c & a & b \\
z & x & y \\
Z & X & Y
\end{matrix}
\right)\nonumber\\
&&=\begin{cases}
-\frac{1}{4},  &\text{if $(x,y,z)\in(a,b,c)\mathbb{R}+\mathbb{Z}^{3}$},\\
0,  &\text{otherwise},
\end{cases}
\end{eqnarray}
where $a,b,c\in\mathbb{N}$ with no common factor, $x,y,z\in\mathbb{R}$, $X,Y,Z$ are three nonzero variables such that $X+Y+Z=0$,
\begin{equation*}
\Omega\left(
\begin{matrix}
a & b & c \\
x & y & z \\
X & Y & Z
\end{matrix}
\right)=\sum_{m,n\geq0}\frac{1}{m!n!}S_{m,n}\left(
\begin{matrix}
a & b & c \\
x & y & z
\end{matrix}
\right)\biggl(\frac{X}{a}\biggl)^{m-1}\biggl(\frac{Y}{b}\biggl)^{n-1}.
\end{equation*}

A natural question is to determine whether or not there exist some another reciprocity formulas for the above-mentioned
Dedekind sums. To do so, we study the generalized Dedekind-Rademacher sums \eqref{eq1.13} from another perspective. We establish a formula of the products of two Bernoulli functions (see Theorem \ref{thm3.1} below). The proof of Theorem \ref{thm3.1} mainly relies on the manipulation of Parseval's formula, Hurwitz's formula, and Lerch's functional equation. As applications of Theorem \ref{thm3.1}, we include some of the previous results as specialisations, among others, the reciprocity formulas \eqref{eq1.7} and \eqref{eq1.12}, and Carlitz's \cite{carlitz4,carlitz5}, Hall and Wilson's \cite{hall1} reciprocity theorems.

This paper is organized as follows. In the second section, we give some auxiliary results. In the third section, we provide the detailed proof of the formula of the products of two Bernoulli functions, and present some special cases as well as immediate consequences of the result. The fourth section concentrates on the feature that has contributed to some more general reciprocity formulas for the generalized Dedekind-Rademacher sums \eqref{eq1.13}.

\section{Some auxiliary results}

For convenience, in the following we always denote by $\mathrm{i}$ the square root of $-1$ such that $\mathrm{i}^{2}=-1$, $\Gamma(s)$ the gamma function defined on $s\in\mathbb{C}$, $[x]$ the floor function (also called the greatest integer function) defined for $x\in\mathbb{R}$ by
\begin{equation*}
[x]=x-\{x\},
\end{equation*}
and $\binom{\alpha}{k}$ the binomial coefficients defined for $k\in\mathbb{Z}$, $\alpha\in\mathbb{C}$ by $\binom{\alpha}{0}=1$ and
\begin{equation*}
\binom{\alpha}{-k}=0,\quad\binom{\alpha}{k}=\frac{\alpha(\alpha-1)(\alpha-2)\cdots(\alpha-k+1)}{k!}\quad(k\geq1).
\end{equation*}
We also write, for $l,k\in\mathbb{N}_{0}$, $\delta_{l,k}$ as the Kronecker delta function given by $\delta_{l,k}=1$ or $0$ according to $l=k$ or $l\not=k$; for $x\in\mathbb{R}$, $\mathrm{sgn}(x)$ as the sign of $x\in\mathbb{R}$ given by
\begin{equation*}
\mathrm{sgn}(x)=\begin{cases}
0,  &\text{if $x=0$},\\
\frac{x}{|x|},  &\text{if $x\not=0$};
\end{cases}
\end{equation*}
for a complex-valued function $f(z)$ on $\mathbb{C}$, $\operatorname{res}_{z=z_{0}}f(z)$ as the residue of $f(z)$ at $z=z_{0}$. For the sake of convergence, the sum
\begin{equation*}
\sum_{d=-\infty}^{+\infty}\frac{1}{d+a}\quad(a\not\in\mathbb{Z})
\end{equation*}
is interpreted as
\begin{equation*}
\lim_{N\rightarrow\infty}\sum_{d=-N}^{N}\frac{1}{d+a}.
\end{equation*}
This ensures that the Bernoulli function $\overline{B}_{n}(x)$ defined in \eqref{eq1.8} can be given for $n\in\mathbb{N}$, $x\in\mathbb{R}$ by the Fourier series (see, e.g., \cite[Theorem 12.19]{apostol2} or \cite[Equation (2)]{hall2}),
\begin{equation}\label{eq2.1}
\overline{B}_{n}(x)=-\frac{n!}{(2\pi\mathrm{i})^{n}}\sideset{}{'}\sum_{k=-\infty}^{+\infty}\frac{e^{2\pi\mathrm{i}kx}}{k^{n}},
\end{equation}
where the dash denotes throughout that undefined terms are excluded from the sum. We now state the following auxiliary results.

\begin{lemma}\label{lem2.1} (Parseval's formula) Suppose that $F(\theta)$ and $G(\theta)$ are two Riemann integrable, complex-valued functions on $\mathbb{R}$ of period $2\pi$ with the Fourier series
\begin{equation*}
F(\theta)=\sum_{n=-\infty}^{+\infty}a_{n}e^{\mathrm{i}n\theta},
\end{equation*}
and
\begin{equation*}
G(\theta)=\sum_{n=-\infty}^{+\infty}b_{n}e^{\mathrm{i}n\theta}.
\end{equation*}
Then
\begin{equation}\label{eq2.2}
\sum_{n=-\infty}^{+\infty}a_{n}\overline{b_{n}}=\frac{1}{2\pi}\int_{0}^{2\pi}F(\theta)\overline{G(\theta)}d\theta,
\end{equation}
where the horizontal bars indicate complex conjugation.
\end{lemma}

\begin{proof}
See \cite[Proposition 3.1.10]{cohen} or \cite[p. 81]{stein} for details.
\end{proof}

\begin{lemma}\label{lem2.2} Let $m,n\in\mathbb{N}$, $d,x,y\in\mathbb{R}$ with $d\not=x\not=y$. Then
\begin{eqnarray}\label{eq2.3}
\frac{1}{(d-x)^{m}(d-y)^{n}}&=&\sum_{j=1}^{m}\binom{m+n-j-1}{n-1}\frac{(-1)^{m-j}}{(x-y)^{m+n-j}(d-x)^{j}}\nonumber\\
&&+\sum_{j=1}^{n}\binom{m+n-j-1}{m-1}\frac{(-1)^{n-j}}{(y-x)^{m+n-j}(d-y)^{j}}.
\end{eqnarray}
\end{lemma}

\begin{proof}
By the partial fraction decomposition, we have
\begin{equation}\label{eq2.4}
\frac{1}{(d-x)^{m}(d-y)^{n}}
=\sum_{j=1}^{m}\frac{M_{m-j}}{(d-x)^{j}}+\sum_{j=1}^{n}\frac{N_{n-j}}{(d-y)^{j}}.
\end{equation}
If we multiply both sides of \eqref{eq2.4} by $(d-y)^{n}$ then we have
\begin{equation}\label{eq2.5}
\frac{1}{(d-x)^{m}}=N_{n-1}(d-y)^{n-1}+\cdots+N_{1}(d-y)+N_{0}+O\bigl((d-y)^{n}\bigl).
\end{equation}
Trivially, taking $d\rightarrow y$ in \eqref{eq2.5} gives
\begin{equation*}
N_{0}=\frac{1}{(y-x)^{m}}.
\end{equation*}
If we make the $(n-j)$-th derivation on both sides of \eqref{eq2.5} with respect to $d$, and let $d\rightarrow y$, then we get that for $j\in\mathbb{N}$ with $1\leq j\leq n$,
\begin{eqnarray}\label{eq2.6}
N_{n-j}&=&\frac{1}{(n-j)!}\frac{\partial^{n-j}}{\partial d^{n-j}}(d-x)^{-m}\biggl|_{d=y}\nonumber\\
&=&(-1)^{n-j}\binom{m+n-j-1}{m-1}\frac{1}{(y-x)^{m+n-j}}.
\end{eqnarray}
Similarly, multiplying both sides of \eqref{eq2.4} by $(d-x)^{m}$ gives
\begin{equation}\label{eq2.7}
\frac{1}{(d-y)^{n}}=M_{m-1}(d-x)^{m-1}+\cdots+M_{1}(d-x)+M_{0}+O\bigl((d-x)^{m}\bigl).
\end{equation}
Hence, by making the $(m-j)$-th derivation on both sides of \eqref{eq2.7} with respect to $d$, and taking $d\rightarrow x$, we see that for $j\in\mathbb{N}$ with $1\leq j\leq m$,
\begin{eqnarray}\label{eq2.8}
M_{m-j}&=&\frac{1}{(m-j)!}\frac{\partial^{m-j}}{\partial d^{m-j}}(d-y)^{-n}\biggl|_{d=x}\nonumber\\
&=&(-1)^{m-j}\binom{m+n-j-1}{n-1}\frac{1}{(x-y)^{m+n-j}}.
\end{eqnarray}
Inserting \eqref{eq2.6} and \eqref{eq2.8} into \eqref{eq2.4}, the desired result follows immediately.
\end{proof}

\begin{lemma}\label{lem2.3} Let $q,j\in\mathbb{N}$ with $q\geq2$, and let $\theta_{r}$ be a real-valued function defined on $r\in\mathbb{N}$ such that $\theta_{r}\not=0,\pm q,\pm 2q,\ldots$. Then
\begin{equation}\label{eq2.9}
\frac{\partial^{j-1}}{\partial a^{j-1}}\bigl(\cot(\pi a)\bigl)\biggl|_{a=\frac{\theta_{r}}{q}}=\frac{\delta_{1,j}}{\mathrm{i}}+2^{j}\pi^{j-1}\mathrm{i}^{j-2}F\biggl(\frac{\theta_{r}}{q},1-j\biggl),
\end{equation}
where $F(x,s)$ is the periodic zeta function given for $x\in\mathbb{R}$, $s\in\mathbb{C}$ by 
\begin{equation*}
F(x,s)=\sum_{n=1}^{\infty}\frac{e^{2\pi \mathrm{i}nx}}{n^{s}}\quad(\Re(s)>1).
\end{equation*}
\end{lemma}

\begin{proof}
See \cite[Equation (2.28)]{he} for details.
\end{proof}

\begin{lemma}\label{lem2.4} Let $j\in\mathbb{N}$, $b,r\in\mathbb{Z}$ with $b\not=0$. Then
\begin{equation}\label{eq2.10}
\sideset{}{'}\sum_{d=-\infty}^{+\infty}\frac{1}{\bigl(d+\frac{r}{b}\bigl)^{j}}
=\frac{(-1)^{j-1}(2\pi\mathrm{i})^{j}\mathrm{sgn}(b)}{j!b^{1-j}}\sum_{l=1}^{|b|}e^{\frac{2\pi \mathrm{i}lr}{b}}\overline{B}_{j}\biggl(\frac{l}{b}\biggl).
\end{equation}
\end{lemma}

\begin{proof}
We first consider the case $b\nmid r$. Since $\cot(a)$ has the following expression in partial fractions (see, e.g., \cite[p. 75]{abramowitz} or \cite[p. 327]{remmert})
\begin{eqnarray}\label{eq2.11}
\cot(a)&=&\frac{1}{a}+2a\sum_{d=1}^{\infty}\frac{1}{a^{2}-d^{2}\pi^{2}}\nonumber\\
&=&\sum_{d=-\infty}^{+\infty}\frac{1}{a+d\pi}\quad(a\not=0,\pm \pi,\pm 2\pi,\ldots),
\end{eqnarray}
we discover from \eqref{eq2.11} that for $j\in\mathbb{N}$,
\begin{equation}\label{eq2.12}
\frac{\partial^{j-1}}{\partial a^{j-1}}\bigl(\cot(\pi a)\bigl)\biggl|_{a=\frac{r}{b}}=\frac{(-1)^{j-1}(j-1)!}{\pi}\sum_{d=-\infty}^{+\infty}\frac{1}{\bigl(d+\frac{r}{b}\bigl)^{j}}.
\end{equation}
Taking $\theta_{r}=qr/b$ in Lemma \ref{lem2.3}, and it then follows from \eqref{eq2.12} that for $j\in\mathbb{N}$,
\begin{equation}\label{eq2.13}
\sum_{d=-\infty}^{+\infty}\frac{1}{\bigl(d+\frac{r}{b}\bigl)^{j}}=-\delta_{1,j}\pi\mathrm{i}+\frac{(-1)^{j}(2\pi\mathrm{i})^{j}}{(j-1)!}F\biggl(\frac{r}{b},1-j\biggl).
\end{equation}
Notice that for $j\in\mathbb{N}_{0}$ (see, e.g., \cite[Theorem 12.13]{apostol2}),
\begin{equation}\label{eq2.14}
\zeta(-j,x)=-\frac{B_{j+1}(x)}{j+1},
\end{equation}
where $\zeta(s,x)$ is the Hurwitz zeta function given for $s\in\mathbb{C}$, $x\in\mathbb{R}$ with $0<x\leq1$ by
\begin{equation*}
\zeta(s,x)=\sum_{n=0}^{\infty}\frac{1}{(n+x)^{s}}\quad(\Re(s)>1).
\end{equation*}
And from \eqref{eq2.1} and $\overline{B}_{0}(x)=1$ we have
\begin{equation}\label{eq2.15}
\overline{B}_{j}(-x)=(-1)^{j}\overline{B}_{j}(x)\quad(j\in\mathbb{N}_{0},x\in\mathbb{R}).
\end{equation}
Therefore, we obtain from the familiar division algorithm stated in \cite[Theorem 1.14]{apostol2}, $B_{j}(1)=B_{j}+\delta_{1,j}$ for $j\in\mathbb{N}_{0}$ described in \cite[Theorem 12.14]{apostol2}, \eqref{eq2.14} and \eqref{eq2.15} that for $j\in\mathbb{N}$,
\begin{eqnarray}\label{eq2.16}
F\biggl(\frac{r}{b},1-j\biggl)&=&\sum_{l=1}^{|b|}\sum_{m=0}^{\infty}\frac{e^{\frac{2\pi \mathrm{i}(m|b|+l)r}{b}}}{(m|b|+l)^{1-j}}\nonumber\\
&=&-\frac{1}{j|b|^{1-j}}\sum_{l=1}^{|b|}e^{\frac{2\pi \mathrm{i}lr}{b}}B_{j}\biggl(\frac{l}{|b|}\biggl)\nonumber\\
&=&-\frac{1}{j|b|^{1-j}}\sum_{l=1}^{|b|}e^{\frac{2\pi \mathrm{i}lr}{b}}\biggl(\overline{B}_{j}\biggl(\frac{l}{|b|}\biggl)+\frac{1}{2}\delta_{1,j}\delta_{l,|b|}\biggl)\nonumber\\
&=&-\frac{\mathrm{sgn}(b)}{jb^{1-j}}\sum_{l=1}^{|b|}e^{\frac{2\pi \mathrm{i}lr}{b}}\overline{B}_{j}\biggl(\frac{l}{b}\biggl)-\frac{1}{2}\delta_{1,j}.
\end{eqnarray}
Inserting \eqref{eq2.16} into \eqref{eq2.13}, we conclude that Lemma \ref{lem2.4} holds true for the case $b\nmid r$. We next consider the case $b\mid r$. Obviously, in this case, for $j\in\mathbb{N}$,
\begin{equation}\label{eq2.17}
\sideset{}{'}\sum_{d=-\infty}^{+\infty}\frac{1}{\bigl(d+\frac{r}{b}\bigl)^{j}}=
\begin{cases}
2\zeta(j),  &2\mid j,\\
0,  &2\nmid j.
\end{cases}
\end{equation}
We know from \cite[Theorems 12.16 and 12.17]{apostol2} that for $j\in\mathbb{N}$,
\begin{equation}\label{eq2.18}
B_{2j+1}=0,
\end{equation}
and
\begin{equation}\label{eq2.19}
\zeta(2j)=\frac{(-1)^{j+1}2^{2j-1}\pi^{2j}B_{2j}}{(2j)!}.
\end{equation}
By applying \eqref{eq2.18} and \eqref{eq2.19} to the right hand side of \eqref{eq2.17}, we see that for $j\in\mathbb{N}$,
\begin{eqnarray}\label{eq2.20}
\sideset{}{'}\sum_{d=-\infty}^{+\infty}\frac{1}{\bigl(d+\frac{r}{b}\bigl)^{j}}&=&-\delta_{1,j}\pi\mathrm{i}-\frac{(2\pi\mathrm{i})^{j}B_{j}}{j!}\nonumber\\
&=&-\frac{(2\pi\mathrm{i})^{j}\overline{B}_{j}(0)}{j!}.
\end{eqnarray}
It is well known that the Bernoulli function satisfy the so-called ``distribution property'' (Raabe's formula \cite{raabe})
\begin{equation}\label{eq2.21}
\overline{B}_{j}(qx)=q^{j-1}\sum_{l=0}^{q-1}\overline{B}_{j}\biggl(x+\frac{l}{q}\biggl),
\end{equation}
where $q\in\mathbb{N}$, $j\in\mathbb{N}_{0}$, $x\in\mathbb{R}$. A direct application of \eqref{eq2.15} and \eqref{eq2.21} says that for $j\in\mathbb{N}_{0}$, $b\in\mathbb{Z}$, $x\in\mathbb{R}$ with $b\not=0$,
\begin{eqnarray}\label{eq2.22}
\sum_{l=1}^{|b|}\overline{B}_{j}\biggl(\frac{l-x}{b}\biggl)&=&\sum_{l=0}^{|b|-1}\overline{B}_{j}\biggl(\frac{|b|-l-x}{b}\biggl)\nonumber\\
&=&(-1)^{j}b^{1-j}\mathrm{sgn}(b)\overline{B}_{j}(x).
\end{eqnarray}
Thus, by applying the case $x=0$ in \eqref{eq2.22} to the right hand side of \eqref{eq2.20}, we see that Lemma \ref{lem2.4} holds true for the case $b\mid r$. This completes the proof of Lemma \ref{lem2.4}.
\end{proof}

\begin{lemma}\label{lem2.5}(Hurwitz's formula) If $0<x\leq1$ and $\Re(s)>1$, or if $0<x<1$ and $\Re(s)>0$, then
\begin{equation}\label{eq2.23}
\zeta(1-s,x)=\frac{\Gamma(s)}{(2\pi)^{s}}\bigl(e^{-\frac{\pi \mathrm{i}s}{2}}F(x,s)+e^{\frac{\pi \mathrm{i}s}{2}}F(-x,s)\bigl),
\end{equation}
where $F(x,s)$ is as in \eqref{eq2.9}, $\zeta(x,s)$ is as in \eqref{eq2.14}.
\end{lemma}

\begin{proof}
See \cite[Theorem 12.6]{apostol2} for details.
\end{proof}

\begin{lemma}\label{lem2.6}(Lerch's functional equation) If $0<x<1$ and $0<a\leq1$ then for all $s\in\mathbb{C}$,
\begin{equation}\label{eq2.24}
\phi(x,a,1-s)=\frac{\Gamma(s)}{(2\pi)^{s}}\bigl(e^{\frac{\pi \mathrm{i}s}{2}-2\pi\mathrm{i}ax}\phi(-a,x,s)+e^{-\frac{\pi \mathrm{i}s}{2}+2\pi\mathrm{i}a(1-x)}\phi(a,1-x,s)\bigl),
\end{equation}
where $\phi(x,a,s)$ is the Lerch zeta function given for $x,a\in\mathbb{R}$, $s\in\mathbb{C}$ with $0<a\leq1$ by
\begin{equation*}
\phi(x,a,s)=\sum_{n=0}^{\infty}\frac{e^{2\pi \mathrm{i}nx}}{(n+a)^{s}}\quad(\Re(s)>1).
\end{equation*}
\end{lemma}

\begin{proof}
See \cite{berndt1,lerch} for details.
\end{proof}

\begin{lemma}\label{lem2.7} Let $j\in\mathbb{N}$, $b,r\in\mathbb{Z}$ with $b\not=0$. Then, for $x\in\mathbb{R}$,
\begin{equation}\label{eq2.25}
\sideset{}{'}\sum_{d=-\infty}^{+\infty}\frac{e^{2\pi \mathrm{i}dx}}{\bigl(d+\frac{r}{b}\bigl)^{j}}
=\frac{(-1)^{j-1}(2\pi\mathrm{i})^{j}\mathrm{sgn}(b)}{j!b^{1-j}}\sum_{l=1}^{|b|}e^{\frac{2\pi \mathrm{i}(l-x)r}{b}}\overline{B}_{j}\biggl(\frac{l-x}{b}\biggl).
\end{equation}
\end{lemma}

\begin{proof} We first consider the case $x\in\mathbb{Z}$. Since $l-x$ runs over a complete residue system modulo $|b|$ as
$l$ does, we see from Lemma \ref{lem2.4} that the case $x\in\mathbb{Z}$ in Lemma \ref{lem2.7} is complete. We next discuss the case $x\not\in\mathbb{Z}$. It suffices to prove that Lemma \ref{lem2.7} holds true in the case when $0<x<1$. It is easily seen that if $b\mid r$ then
\begin{eqnarray}\label{eq2.26}
\sideset{}{'}\sum_{d=-\infty}^{+\infty}\frac{e^{2\pi \mathrm{i}dx}}{\bigl(d+\frac{r}{b}\bigl)^{j}}
&=&e^{-\frac{2\pi \mathrm{i}rx}{b}}\biggl(\sum_{d=1}^{\infty}\frac{e^{2\pi \mathrm{i}dx}}{d^{j}}+(-1)^{j}\sum_{d=1}^{\infty}\frac{e^{-2\pi \mathrm{i}dx}}{d^{j}}\biggl)\nonumber\\
&=&e^{-\frac{2\pi \mathrm{i}rx}{b}}\bigl((F(x,j)+(-1)^{j}F(-x,j)\bigl),
\end{eqnarray}
where $F(x,s)$ is as in \eqref{eq2.9}. Taking $s=j$ in Lemma \ref{lem2.5}, and using \eqref{eq2.14}, we discover that \eqref{eq2.26} can be rewritten as
\begin{eqnarray*}
\sideset{}{'}\sum_{d=-\infty}^{+\infty}\frac{e^{2\pi \mathrm{i}dx}}{\bigl(d+\frac{r}{b}\bigl)^{j}}
&=&e^{-\frac{2\pi \mathrm{i}rx}{b}}\frac{(2\pi\mathrm{i})^{j}\zeta(1-j,x)}{(j-1)!}\\
&=&-e^{-\frac{2\pi \mathrm{i}rx}{b}}\frac{(2\pi\mathrm{i})^{j}\overline{B}_{j}(x)}{j!},
\end{eqnarray*}
from which and \eqref{eq2.22} we conclude that Lemma \ref{lem2.7} holds true in the case when $x\not\in\mathbb{Z}$ and $b\mid r$. We now consider the case $x\not\in\mathbb{Z}$ and $b\nmid r$.
Clearly, in this case, we have
\begin{eqnarray}\label{eq2.27}
&&\sideset{}{'}\sum_{d=-\infty}^{+\infty}\frac{e^{2\pi \mathrm{i}dx}}{\bigl(d+\frac{r}{b}\bigl)^{j}}\nonumber\\
&&=e^{-2\pi \mathrm{i}[\frac{r}{b}]x}\biggl(\sum_{d=0}^{\infty}\frac{e^{2\pi \mathrm{i}dx}}{(d+\{\frac{r}{b}\})^{j}}+(-1)^{j}e^{-2\pi \mathrm{i}x}\sum_{d=0}^{\infty}\frac{e^{-2\pi \mathrm{i}dx}}{(d+1-\{\frac{r}{b}\})^{j}}\biggl)\nonumber\\
&&=e^{-2\pi \mathrm{i}[\frac{r}{b}]x}\biggl(\phi\biggl(x,\biggl\{\frac{r}{b}\biggl\},j\biggl)+(-1)^{j}e^{-2\pi \mathrm{i}x}\phi\biggl(-x,1-\biggl\{\frac{r}{b}\biggl\},j\biggl)\biggl),
\end{eqnarray}
where $\phi(x,a,s)$ is as in \eqref{eq2.24}.
By taking $x=1-\{\frac{r}{b}\}$ and $s=j$, and then replacing $a$ by $x$ in Lemma \ref{lem2.6}, we find that
\begin{eqnarray}\label{eq2.28}
&&\phi\biggl(1-\biggl\{\frac{r}{b}\biggl\},x,1-j\biggl)\nonumber\\
&&=\frac{(j-1)!}{(2\pi)^{j}}\biggl(e^{\frac{\pi \mathrm{i}j}{2}-2\pi\mathrm{i}x(1-\{\frac{r}{b}\})}\phi\biggl(-x,1-\biggl\{\frac{r}{b}\biggl\},j\biggl)\nonumber\\
&&\quad+e^{-\frac{\pi \mathrm{i}j}{2}+2\pi\mathrm{i}x\{\frac{r}{b}\}}\phi\biggl(x,\biggl\{\frac{r}{b}\biggl\},j\biggl)\biggl).
\end{eqnarray}
Multiplying both sides of \eqref{eq2.28} by $e^{\frac{\pi \mathrm{i}j}{2}-2\pi\mathrm{i}\{\frac{r}{b}\}x}$, we know from \eqref{eq2.27} that
\begin{equation}\label{eq2.29}
\sideset{}{'}\sum_{d=-\infty}^{+\infty}\frac{e^{2\pi \mathrm{i}dx}}{\bigl(d+\frac{r}{b}\bigl)^{j}}
=e^{-\frac{2\pi \mathrm{i}rx}{b}}\frac{(2\pi \mathrm{i})^{j}\phi(1-\{\frac{r}{b}\},x,1-j)}{(j-1)!}.
\end{equation}
It is easy from the division algorithm, \eqref{eq2.14} and \eqref{eq2.15} to see that
\begin{eqnarray}\label{eq2.30}
\phi\biggl(1-\biggl\{\frac{r}{b}\biggl\},x,1-j\biggl)&=&\sum_{l=0}^{|b|-1}\sum_{m=0}^{\infty}\frac{e^{-\frac{2\pi \mathrm{i}(m|b|+l)r}{b}}}{(m|b|+l+x)^{1-j}}\nonumber\\
&=&-\frac{1}{j|b|^{1-j}}\sum_{l=0}^{|b|-1}e^{-\frac{2\pi \mathrm{i}lr}{b}}\overline{B}_{j}\biggl(\frac{l+x}{|b|}\biggl)\nonumber\\
&=&-\frac{1}{j|b|^{1-j}}\sum_{l=1}^{|b|}e^{-\frac{2\pi \mathrm{i}(|b|-l)r}{b}}\overline{B}_{j}\biggl(\frac{|b|-l+x}{|b|}\biggl)\nonumber\\
&=&\frac{(-1)^{j-1}\mathrm{sgn}(b)}{jb^{1-j}}\sum_{l=1}^{|b|}e^{\frac{2\pi \mathrm{i}lr}{b}}\overline{B}_{j}\biggl(\frac{l-x}{b}\biggl).
\end{eqnarray}
Therefore, inserting \eqref{eq2.30} into \eqref{eq2.29}, we say that Lemma \ref{lem2.7} holds true for the case $x\not\in\mathbb{Z}$ and $b\nmid r$. This completes the proof of Lemma \ref{lem2.7}.
\end{proof}

\begin{lemma}\label{lem2.8} Let $m,n\in\mathbb{N}$. Then
\begin{equation}\label{eq2.31}
\sum_{j=1}^{m}\binom{m+n-j-1}{n-1}=\binom{m+n-1}{n}.
\end{equation}
\end{lemma}

\begin{proof}
It is easily seen from the formal binomial series stated in \cite[p. 37]{comtet} that for $n\in\mathbb{N}_{0}$, $\alpha\in\mathbb{C}$,
\begin{equation}\label{eq2.32}
\binom{\alpha}{n}=\operatorname{res}_{z=0}(1+z)^{\alpha}z^{-n-1}.
\end{equation}
Hence, we get from \eqref{eq2.32} that
\begin{eqnarray*}
\sum_{j=1}^{m}\binom{m+n-j-1}{n-1}&=&\sum_{j=1}^{m}\operatorname{res}_{z=0}(1+z)^{m+n-j-1}z^{-n}\nonumber\\
&=&\operatorname{res}_{z=0}z^{-n}(1+z)^{n-1}\sum_{j=1}^{m}(1+z)^{m-j}\nonumber\\
&=&\operatorname{res}_{z=0}z^{-n-1}\bigl((1+z)^{m+n-1}-(1+z)^{n-1}\bigl)\nonumber\\
&=&\binom{m+n-1}{n},
\end{eqnarray*}
as desired. This concludes the proof of Lemma \ref{lem2.8}.
\end{proof}

\section{Formulas of products of Bernoulli functions}

We are now in a position to give the following formula of the products of two Bernoulli functions.

\begin{theorem}\label{thm3.1} Let $m,n\in\mathbb{N}$, $a,b\in\mathbb{Z}$ with $a\not=0$ and $b\not=0$. Then, for $x,y,z\in\mathbb{R}$,
\begin{eqnarray}\label{eq3.1}
&&\overline{B}_{m}(ax+y)\overline{B}_{n}(bx+z)\nonumber\\
&&=nb^{n-1}\mathrm{sgn}(b)\sum_{j=0}^{m}\binom{m}{j}\frac{(-1)^{j}a^{m-j}}{m+n-j}\sum_{l=1}^{|b|}\overline{B}_{j}\biggl(\frac{a(l+z)}{b}-y\biggl)\nonumber\\
&&\qquad\times\overline{B}_{m+n-j}\biggl(x+\frac{l+z}{b}\biggl)\nonumber\\
&&\quad+ma^{m-1}\mathrm{sgn}(a)\sum_{j=0}^{n}\binom{n}{j}\frac{(-1)^{j}b^{n-j}}{m+n-j}\sum_{l=1}^{|a|}\overline{B}_{j}\biggl(\frac{b(l+y)}{a}-z\biggl)\nonumber\\
&&\qquad\times\overline{B}_{m+n-j}\biggl(x+\frac{l+y}{a}\biggl)\nonumber\\
&&\quad+\frac{(-1)^{n-1}m!n!(a,b)^{m+n}\overline{B}_{m+n}\bigl(\frac{by-az}{(a,b)}\bigl)}{a^{n}b^{m}(m+n)!}\nonumber\\
&&\quad-\delta_{1,m}\delta_{1,n}\mathrm{sgn}(ab)\frac{\delta_{\mathbb{Z}}(ax+y)\delta_{\mathbb{Z}}(bx+z)}{4},
\end{eqnarray}
where $\delta_{\mathbb{Z}}(x)$ is as in \eqref{eq1.5}, $(a,b)$ is as in \eqref{eq1.7}.
\end{theorem}

\begin{proof}
Let
\begin{equation*}
f(x)=\overline{B}_{m}(ax+y)\overline{B}_{n}(bx+z).
\end{equation*}
Since $f(x)$ is of bounded variation on every finite interval, $f(x)$ may be expanded in a Fourier series
\begin{equation}\label{eq3.2}
\frac{f(x^{+})+f(x^{-})}{2}=\sum_{k=-\infty}^{+\infty}c_{k}(m,n|a,b,y,z)e^{2\pi\mathrm{i}kx},
\end{equation}
where the Fourier coefficients $c_{k}(m,n|a,b,y,z)$ are determined by
\begin{equation}\label{eq3.3}
c_{k}(m,n|a,b,y,z)=\int_{0}^{1}\overline{B}_{m}(ax+y)\overline{B}_{n}(bx+z)e^{-2\pi\mathrm{i}kx}dx.
\end{equation}
We now make the change of the variable $\theta=2\pi x$ in Lemma \ref{lem2.1} and set
\begin{equation*}
F(2\pi x)=\overline{B}_{m}(ax+y),\quad\overline{G(2\pi x)}=\overline{B}_{n}(bx+z)e^{-2\pi\mathrm{i}kx},
\end{equation*}
it then follows from \eqref{eq2.1} and \eqref{eq2.2} that we rewrite \eqref{eq3.3} as
\begin{equation}\label{eq3.4}
c_{k}(m,n|a,b,y,z)=\frac{m!n!}{(2\pi \mathrm{i})^{m+n}}\underset{al+bj=k}{\sideset{}{'}\sum_{l=-\infty}^{+\infty}\sideset{}{'}\sum_{j=-\infty}^{+\infty}}\frac{e^{2\pi\mathrm{i}(ly+jz)}} {l^{m}j^{n}}.
\end{equation}
Since the Diophantine equation $ax+by=k$ is solvable if and only if $(a,b)\mid k$, and in the case when it is solvable, all solutions of $ax+by=k$ are given by
\begin{equation*}
x=\frac{k\overline{a}}{(a,b)}+\frac{bd}{(a,b)},\quad y=\frac{k\overline{b}}{(a,b)}-\frac{ad}{(a,b)},
\end{equation*}
where $\overline{a},\overline{b},d\in\mathbb{Z}$ such that $a\overline{a}+b\overline{b}=(a,b)$, we see from \eqref{eq3.4} that if $(a,b)\nmid k$ then
\begin{equation}\label{eq3.5}
c_{k}(m,n|a,b,y,z)=0,
\end{equation}
and if $(a,b)\mid k$ then
\begin{eqnarray}\label{eq3.6}
&&c_{k}(m,n|a,b,y,z)\nonumber\\
&&=\frac{m!n!(a,b)^{m+n}e^{\frac{2\pi\mathrm{i}k(\overline{a}y+\overline{b}z)}{(a,b)}}}{(2\pi \mathrm{i})^{m+n}}\sideset{}{'}\sum_{d=-\infty}^{+\infty}\frac{e^{\frac{2\pi\mathrm{i}d(by-az)}{(a,b)}}}{(k\overline{a}+bd)^{m}(k\overline{b}-ad)^{n}}\nonumber\\
&&=\frac{(-1)^{n}m!n!(a,b)^{m+n}e^{\frac{2\pi\mathrm{i}k(\overline{a}y+\overline{b}z)}{(a,b)}}}{a^{n}b^{m}(2\pi \mathrm{i})^{m+n}}\sideset{}{'}\sum_{d=-\infty}^{+\infty}\frac{e^{\frac{2\pi\mathrm{i}d(by-az)}{(a,b)}}}{\bigl(d+\frac{k\overline{a}}{b}\bigl)^{m}\bigl(d-\frac{k\overline{b}}{a}\bigl)^{n}}.
\end{eqnarray}
In particular, we conclude from \eqref{eq2.1} and \eqref{eq3.6} that
\begin{equation}\label{eq3.7}
c_{0}(m,n|a,b,y,z)=\frac{(-1)^{n-1}m!n!(a,b)^{m+n}\overline{B}_{m+n}\bigl(\frac{by-az}{(a,b)}\bigl)}{a^{n}b^{m}(m+n)!}.
\end{equation}
We next consider the case $k\not=0$ in \eqref{eq3.6}. Note that there is an implicit restriction
\begin{equation*}
\biggl(\frac{a}{(a,b)},\overline{b}\biggl)=\biggl(\frac{b}{(a,b)},\overline{a}\biggl)=1.
\end{equation*}
Hence, by taking $x=-k\overline{a}/b$ and $y=k\overline{b}/a$ in Lemma \ref{lem2.2}, we obtain from \eqref{eq3.6} that if $(a,b)\mid k$ then
\begin{eqnarray}\label{eq3.8}
&&c_{k}(m,n|a,b,y,z)\nonumber\\
&&=\frac{(-1)^{n}m!n!(a,b)^{m+n}e^{\frac{2\pi\mathrm{i}k(\overline{a}y+\overline{b}z)}{(a,b)}}}{a^{n}b^{m}(2\pi \mathrm{i})^{m+n}}\nonumber\\
&&\quad\times\biggl(\sum_{j=1}^{m}\binom{m+n-j-1}{n-1}(-1)^{m-j}\biggl(-\frac{ab}{k(a,b)}\biggl)^{m+n-j}\nonumber\\
&&\qquad\quad\times\biggl(\sideset{}{'}\sum_{d=-\infty}^{+\infty}\frac{e^{\frac{2\pi\mathrm{i}d(by-az)}{(a,b)}}}{\bigl(d+\frac{k\overline{a}}{b}\bigl)^{j}}-\epsilon(a,k)e^{\frac{2\pi\mathrm{i}k\overline{b}(by-az)}{a(a,b)}}\biggl(\frac{ab}{k(a,b)}\biggl)^{j}\biggl)\nonumber\\
&&\qquad+\sum_{j=1}^{n}\binom{m+n-j-1}{m-1}(-1)^{n-j}\biggl(\frac{ab}{k(a,b)}\biggl)^{m+n-j}\nonumber\\
&&\qquad\quad\times\biggl(\sideset{}{'}\sum_{d=-\infty}^{+\infty}\frac{e^{\frac{2\pi\mathrm{i}d(by-az)}{(a,b)}}}{\bigl(d-\frac{k\overline{b}}{a}\bigl)^{j}}-\epsilon(b,k)e^{-\frac{2\pi\mathrm{i}k\overline{a}(by-az)}{b(a,b)}}\biggl(-\frac{ab}{k(a,b)}\biggl)^{j}\biggl)\biggl),
\end{eqnarray}
where $\epsilon(a,k)=1$ or $0$ according to $a\mid k$ or $a\nmid k$. With the help of Lemma \ref{lem2.7}, we discover that \eqref{eq3.8} can be rewritten in the following way
\begin{eqnarray}\label{eq3.9}
&&c_{k}(m,n|a,b,y,z)\nonumber\\
&&=-b^{n-1}m!n!(a,b)\sum_{j=1}^{m}\binom{m+n-j-1}{n-1}\frac{(-1)^{j}a^{m-j}}{j!(2\pi\mathrm{i}k)^{m+n-j}}\nonumber\\
&&\qquad\times\mathrm{sgn}\biggl(\frac{b}{(a,b)}\biggl)\sum_{l=1}^{|\frac{b}{(a,b)}|}e^{\frac{2\pi\mathrm{i}k(z+\overline{a}l)}{b}}\overline{B}_{j}\biggl(\frac{(a,b)l-(by-az)}{b}\biggl)\nonumber\\
&&\quad-\epsilon(a,k)\frac{a^{m}b^{n}m!n!e^{\frac{2\pi\mathrm{i}ky}{a}}}{(2\pi\mathrm{i}k)^{m+n}}\sum_{j=1}^{m}\binom{m+n-j-1}{n-1}\nonumber\\
&&\quad -a^{m-1}m!n!(a,b)\sum_{j=1}^{n}\binom{m+n-j-1}{m-1}\frac{b^{n-j}}{j!(2\pi\mathrm{i}k)^{m+n-j}}\nonumber\\
&&\qquad\times\mathrm{sgn}\biggl(\frac{a}{(a,b)}\biggl)\sum_{l=1}^{|\frac{a}{(a,b)}|}e^{\frac{2\pi\mathrm{i}k(y-\overline{b}l)}{a}}\overline{B}_{j}\biggl(\frac{(a,b)l-(by-az)}{a}\biggl)\nonumber\\
&&\quad-\epsilon(b,k)\frac{a^{m}b^{n}m!n!e^{\frac{2\pi\mathrm{i}kz}{b}}}{(2\pi\mathrm{i}k)^{m+n}}\sum_{j=1}^{n}\binom{m+n-j-1}{m-1}.
\end{eqnarray}
It follows from \eqref{eq3.9} and Lemma \ref{lem2.8} that if $(a,b)\mid k$ then
\begin{eqnarray}\label{eq3.10}
c_{k}(m,n|a,b,y,z)&=&-nb^{n-1}(a,b)\mathrm{sgn}(b)\sum_{j=1}^{m}\binom{m}{j}\frac{(-1)^{j}a^{m-j}(m+n-j-1)!}{(2\pi\mathrm{i}k)^{m+n-j}}\nonumber\\
&&\quad\times\sum_{l=1}^{|\frac{b}{(a,b)}|}e^{\frac{2\pi\mathrm{i}k(z+\overline{a}l)}{b}}\overline{B}_{j}\biggl(\frac{(a,b)l-(by-az)}{b}\biggl)\nonumber\\
&&-ma^{m-1}(a,b)\mathrm{sgn}(a)\sum_{j=1}^{n}\binom{n}{j}\frac{b^{n-j}(m+n-j-1)!}{(2\pi\mathrm{i}k)^{m+n-j}}\nonumber\\
&&\quad\times\sum_{l=1}^{|\frac{a}{(a,b)}|}e^{\frac{2\pi\mathrm{i}k(y-\overline{b}l)}{a}}\overline{B}_{j}\biggl(\frac{(a,b)l-(by-az)}{a}\biggl)\nonumber\\
&&-\epsilon(a,k)e^{\frac{2\pi\mathrm{i}ky}{a}}\frac{m(m+n-1)!a^{m}b^{n}}{(2\pi\mathrm{i}k)^{m+n}}\nonumber\\
&&-\epsilon(b,k)e^{\frac{2\pi\mathrm{i}kz}{b}}\frac{n(m+n-1)!a^{m}b^{n}}{(2\pi\mathrm{i}k)^{m+n}}.
\end{eqnarray}
Since $\bigl(\frac{a}{(a,b)},\frac{b}{(a,b)}\bigl)=1$, $\frac{al}{(a,b)}$ runs over a complete residue
system modulo $|\frac{b}{(a,b)}|$ as $l$ does, and $\frac{-bl}{(a,b)}$ runs over a complete residue system modulo $|\frac{a}{(a,b)}|$
as $l$ does. It then follows from \eqref{eq2.15} and \eqref{eq3.10} that if $(a,b)\mid k$ then
\begin{eqnarray}\label{eq3.11}
c_{k}(m,n|a,b,y,z)&=&-nb^{n-1}(a,b)\mathrm{sgn}(b)\sum_{j=1}^{m}\binom{m}{j}\frac{(-1)^{j}a^{m-j}(m+n-j-1)!}{(2\pi\mathrm{i}k)^{m+n-j}}\nonumber\\
&&\quad\times\sum_{l=1}^{|\frac{b}{(a,b)}|}e^{\frac{2\pi\mathrm{i}k(l+z)}{b}}\overline{B}_{j}\biggl(\frac{a(l+z)}{b}-y\biggl)\nonumber\\
&&-ma^{m-1}(a,b)\mathrm{sgn}(a)\sum_{j=1}^{n}\binom{n}{j}\frac{(-1)^{j}b^{n-j}(m+n-j-1)!}{(2\pi\mathrm{i}k)^{m+n-j}}\nonumber\\
&&\quad\times\sum_{l=1}^{|\frac{a}{(a,b)}|}e^{\frac{2\pi\mathrm{i}k(l+y)}{a}}\overline{B}_{j}\biggl(\frac{b(l+y)}{a}-z\biggl)\nonumber\\
&&-\epsilon(a,k)e^{\frac{2\pi\mathrm{i}ky}{a}}\frac{m(m+n-1)!a^{m}b^{n}}{(2\pi\mathrm{i}k)^{m+n}}\nonumber\\
&&-\epsilon(b,k)e^{\frac{2\pi\mathrm{i}kz}{b}}\frac{n(m+n-1)!a^{m}b^{n}}{(2\pi\mathrm{i}k)^{m+n}}.
\end{eqnarray}
Inserting \eqref{eq3.5}, \eqref{eq3.7} and \eqref{eq3.11} into \eqref{eq3.2}, we arrive at
\begin{eqnarray}\label{eq3.12}
&&\overline{B}_{m}(ax+y)\overline{B}_{n}(bx+z)+\delta_{1,m}\delta_{1,n}\mathrm{sgn}(ab)\frac{\delta_{\mathbb{Z}}(ax+y)\delta_{\mathbb{Z}}(bx+z)}{4}\nonumber\\
&&=\sideset{}{'}\sum_{\substack{k=-\infty\\(a,b)\mid k}}^{+\infty}\biggl(-nb^{n-1}(a,b)\mathrm{sgn}(b)\sum_{j=1}^{m}\binom{m}{j}\frac{(-1)^{j}a^{m-j}(m+n-j-1)!}{(2\pi\mathrm{i}k)^{m+n-j}}\nonumber\\
&&\qquad\times\sum_{l=1}^{|\frac{b}{(a,b)}|}e^{\frac{2\pi\mathrm{i}k(l+z)}{b}}\overline{B}_{j}\biggl(\frac{a(l+z)}{b}-y\biggl)\nonumber\\
&&\quad-ma^{m-1}(a,b)\mathrm{sgn}(a)\sum_{j=1}^{n}\binom{n}{j}\frac{(-1)^{j}b^{n-j}(m+n-j-1)!}{(2\pi\mathrm{i}k)^{m+n-j}}\nonumber\\
&&\qquad\times\sum_{l=1}^{|\frac{a}{(a,b)}|}e^{\frac{2\pi\mathrm{i}k(l+y)}{a}}\overline{B}_{j}\biggl(\frac{b(l+y)}{a}-z\biggl)\biggl)e^{2\pi\mathrm{i}kx}\nonumber\\
&&\quad-\sideset{}{'}\sum_{\substack{k=-\infty\\a\mid k}}^{+\infty}\biggl(\frac{m(m+n-1)!a^{m}b^{n}e^{\frac{2\pi\mathrm{i}ky}{a}}}{(2\pi\mathrm{i}k)^{m+n}}\biggl)e^{2\pi\mathrm{i}kx}\nonumber\\
&&\quad-\sideset{}{'}\sum_{\substack{k=-\infty\\b\mid k}}^{+\infty}\biggl(\frac{n(m+n-1)!a^{m}b^{n}e^{\frac{2\pi\mathrm{i}kz}{b}}}{(2\pi\mathrm{i}k)^{m+n}}\biggl)e^{2\pi\mathrm{i}kx}\nonumber\\
&&\quad+\frac{(-1)^{n-1}m!n!(a,b)^{m+n}\overline{B}_{m+n}\bigl(\frac{by-az}{(a,b)}\bigl)}{a^{n}b^{m}(m+n)!}.
\end{eqnarray}
So from \eqref{eq2.1} and \eqref{eq3.12}, we know that
\begin{eqnarray}\label{eq3.13}
&&\overline{B}_{m}(ax+y)\overline{B}_{n}(bx+z)+\delta_{1,m}\delta_{1,n}\mathrm{sgn}(ab)\frac{\delta_{\mathbb{Z}}(ax+y)\delta_{\mathbb{Z}}(bx+z)}{4}\nonumber\\
&&=nb^{n-1}\mathrm{sgn}(b)\sum_{j=1}^{m}\binom{m}{j}\frac{(-1)^{j}a^{m-j}}{(a,b)^{m+n-j-1}(m+n-j)}\nonumber\\
&&\qquad\times\sum_{l=1}^{|\frac{b}{(a,b)}|}\overline{B}_{j}\biggl(\frac{a(l+z)}{b}-y\biggl)\overline{B}_{m+n-j}\biggl((a,b)x+\frac{(a,b)(l+z)}{b}\biggl)\nonumber\\
&&\quad+ma^{m-1}\mathrm{sgn}(a)\sum_{j=1}^{n}\binom{n}{j}\frac{(-1)^{j}b^{n-j}}{(a,b)^{m+n-j-1}(m+n-j)}\nonumber\\
&&\qquad\times\sum_{l=1}^{|\frac{a}{(a,b)}|}\overline{B}_{j}\biggl(\frac{b(l+y)}{a}-z\biggl)\overline{B}_{m+n-j}\biggl((a,b)x+\frac{(a,b)(l+y)}{a}\biggl)\nonumber\\
&&\quad+\frac{mb^{n}\overline{B}_{m+n}(ax+y)}{a^{n}(m+n)}+\frac{na^{m}\overline{B}_{m+n}(bx+z)}{b^{m}(m+n)}\nonumber\\
&&\quad+\frac{(-1)^{n-1}m!n!(a,b)^{m+n}\overline{B}_{m+n}\bigl(\frac{by-az}{(a,b)}\bigl)}{a^{n}b^{m}(m+n)!}.
\end{eqnarray}
It is easily shown from \eqref{eq2.21}, the property of residue systems and the division algorithm that
\begin{eqnarray}\label{eq3.14}
&&\frac{1}{(a,b)^{m+n-j-1}}\sum_{l=1}^{|\frac{b}{(a,b)}|}\overline{B}_{j}\biggl(\frac{a(l+z)}{b}-y\biggl)\overline{B}_{m+n-j}\biggl((a,b)x+\frac{(a,b)(l+z)}{b}\biggl)\nonumber\\
&&=\sum_{k=0}^{(a,b)-1}\sum_{l=1}^{|\frac{b}{(a,b)}|}\overline{B}_{j}\biggl(\frac{a\bigl(l+z+k|\frac{b}{(a,b)}|\bigl)}{b}-y\biggl)\overline{B}_{m+n-j}\biggl(x+\frac{l+z+k|\frac{b}{(a,b)}|}{b}\biggl)\nonumber\\
&&=\sum_{l=1}^{|b|}\overline{B}_{j}\biggl(\frac{a(l+z)}{b}-y\biggl)\overline{B}_{m+n-j}\biggl(x+\frac{l+z}{b}\biggl).
\end{eqnarray}
Therefore, by applying \eqref{eq3.14} to the right hand side of \eqref{eq3.13}, in view of $\overline{B}_{0}(x)=1$, \eqref{eq2.15} and \eqref{eq2.22}, we get \eqref{eq3.1} and finish the proof of Theorem \ref{thm3.1}.
\end{proof}

We next show some special cases of Theorem \ref{thm3.1}. We first present the following result.

\begin{corollary}\label{cor3.2} Let $m,n\in\mathbb{N}$, $a,b\in\mathbb{Z}$ with $a\not=0$ and $b\not=0$. Then, for $x,y,z\in\mathbb{R}$,
\begin{eqnarray}\label{eq3.15}
&&nb^{n-1}\mathrm{sgn}(b)\sum_{j=0}^{m}\binom{m}{j}\frac{(-1)^{m-j}a^{m-j}}{m+n-j}s_{j,m+n-j}(a,b;x,y)\nonumber\\
&&\quad+ma^{m-1}\mathrm{sgn}(a)\sum_{j=0}^{n}\binom{n}{j}\frac{(-1)^{n-j}b^{n-j}}{m+n-j}s_{j,m+n-j}(b,a;y,x)\nonumber\\
&&=-\delta_{1,m}\delta_{1,n}\mathrm{sgn}(ab)\frac{\delta_{\mathbb{Z}}(x)\delta_{\mathbb{Z}}(y)}{4}+\overline{B}_{m}(x)\overline{B}_{n}(y)\nonumber\\
&&\quad+\frac{m!n!(a,b)^{m+n}\overline{B}_{m+n}\bigl(\frac{ay+bx}{(a,b)}\bigl)}{a^{n}b^{m}(m+n)!},
\end{eqnarray}
where $\delta_{\mathbb{Z}}(x)$ is as in \eqref{eq1.5}, $(a,b)$ is as in \eqref{eq1.7}, $s_{m,n}(a,b;x,y)$ is defined for $m,n\in\mathbb{N}_{0}$, $a,b\in\mathbb{Z}$, $x,y\in\mathbb{R}$ with $b\not=0$ by
\begin{equation}\label{eq3.16}
s_{m,n}(a,b;x,y)=\sum_{r=1}^{|b|}\overline{B}_{m}\biggl(\frac{a(r+y)}{b}+x\biggl)\overline{B}_{n}\biggl(\frac{r+y}{b}\biggl).
\end{equation}
\end{corollary}

\begin{proof}
By taking $x\in\mathbb{Z}$ in Theorem \ref{thm3.1}, we find that for $m,n\in\mathbb{N}$, $a,b\in\mathbb{Z}$, $y,z\in\mathbb{R}$ with $a\not=0$ and $b\not=0$,
\begin{eqnarray}\label{eq3.17}
&&nb^{n-1}\mathrm{sgn}(b)\sum_{j=0}^{m}\binom{m}{j}\frac{(-1)^{j}a^{m-j}}{m+n-j}\sum_{l=1}^{|b|}\overline{B}_{j}\biggl(\frac{a(l+z)}{b}-y\biggl)\overline{B}_{m+n-j}\biggl(\frac{l+z}{b}\biggl)\nonumber\\
&&\quad+ma^{m-1}\mathrm{sgn}(a)\sum_{j=0}^{n}\binom{n}{j}\frac{(-1)^{j}b^{n-j}}{m+n-j}\sum_{l=1}^{|a|}\overline{B}_{j}\biggl(\frac{b(l+y)}{a}-z\biggl)\overline{B}_{m+n-j}\biggl(\frac{l+y}{a}\biggl)\nonumber\\
&&=\delta_{1,m}\delta_{1,n}\mathrm{sgn}(ab)\frac{\delta_{\mathbb{Z}}(y)\delta_{\mathbb{Z}}(z)}{4}+\overline{B}_{m}(y)\overline{B}_{n}(z)\nonumber\\
&&\quad+\frac{(-1)^{n}m!n!(a,b)^{m+n}\overline{B}_{m+n}\bigl(\frac{by-az}{(a,b)}\bigl)}{a^{n}b^{m}(m+n)!}.
\end{eqnarray}
If we replace $y$ by $-x$, $z$ by $y$ in \eqref{eq3.17}, then we obtain from \eqref{eq2.15} and the property of residue systems that
\begin{eqnarray}\label{eq3.18}
&&nb^{n-1}\mathrm{sgn}(b)\sum_{j=0}^{m}\binom{m}{j}\frac{(-1)^{j}a^{m-j}}{m+n-j}s_{j,m+n-j}(a,b;x,y)\nonumber\\
&&\quad+(-1)^{m+n}ma^{m-1}\mathrm{sgn}(a)\sum_{j=0}^{n}\binom{n}{j}\frac{(-1)^{j}b^{n-j}}{m+n-j}s_{j,m+n-j}(b,a;y,x)\nonumber\\
&&=\delta_{1,m}\delta_{1,n}\mathrm{sgn}(ab)\frac{\delta_{\mathbb{Z}}(x)\delta_{\mathbb{Z}}(y)}{4}+(-1)^{m}\overline{B}_{m}(x)\overline{B}_{n}(y)\nonumber\\
&&\quad+\frac{(-1)^{m}m!n!(a,b)^{m+n}\overline{B}_{m+n}\bigl(\frac{ay+bx}{(a,b)}\bigl)}{a^{n}b^{m}(m+n)!}.
\end{eqnarray}
Thus, multiplying both sides of \eqref{eq3.18} by $(-1)^{m}$ gives the desired result.
\end{proof}

In particular, taking $m=n=1$ in Corollary \ref{cor3.2}, we conclude from \eqref{eq2.21} that for $a,b\in\mathbb{N}$, $x,y\in\mathbb{R}$,
\begin{eqnarray}\label{eq3.19}
&&s(a,b;x,y)+s(b,a;y,x)\nonumber\\
&&=-\frac{\delta_{\mathbb{Z}}(x)\delta_{\mathbb{Z}}(y)}{4}+((x))((y))+\frac{a\overline{B}_{2}(y)}{2b}+\frac{b\overline{B}_{2}(x)}{2a}\nonumber\\
&&\quad+\frac{(a,b)^{2}\overline{B}_{2}\bigl(\frac{ay+bx}{(a,b)}\bigl)}{2ab},
\end{eqnarray}
from which we get Rademacher's reciprocity formula \eqref{eq1.5} immediately. It should be noted that the formula \eqref{eq3.19} can be also easily deduced from Berndt's reciprocity formula \eqref{eq1.7} (see \cite[Corollary 4.5]{berndt3} for details).

We now use Theorem \ref{thm3.1} to give another formula of the products of two Bernoulli functions as follows.

\begin{theorem}\label{thm3.3} Let $m,n\in\mathbb{N}$, $a,b\in\mathbb{Z}$ with $a\not=0$ and $b\not=0$. Then, for $x,y,z\in\mathbb{R}$,
\begin{eqnarray}\label{eq3.20}
&&ma\overline{B}_{m-1}(ax+y)\overline{B}_{n}(bx+z)+nb\overline{B}_{m}(ax+y)\overline{B}_{n-1}(bx+z)\nonumber\\
&&=nb^{n-1}\mathrm{sgn}(b)\sum_{j=0}^{m}\binom{m}{j}(-1)^{j}a^{m-j}\sum_{l=1}^{|b|}\overline{B}_{j}\biggl(\frac{a(l+z)}{b}-y\biggl)\nonumber\\
&&\qquad\times\overline{B}_{m+n-j-1}\biggl(x+\frac{l+z}{b}\biggl)\nonumber\\
&&\quad+ma^{m-1}\mathrm{sgn}(a)\sum_{j=0}^{n}\binom{n}{j}(-1)^{j}b^{n-j}\sum_{l=1}^{|a|}\overline{B}_{j}\biggl(\frac{b(l+y)}{a}-z\biggl)\nonumber\\
&&\qquad\times\overline{B}_{m+n-j-1}\biggl(x+\frac{l+y}{a}\biggl)\nonumber\\
&&\quad-\mathrm{sgn}(ab)\delta_{\mathbb{Z}}(ax+y)\delta_{\mathbb{Z}}(bx+z)\frac{\delta_{1,m-1}\delta_{1,n}ma+\delta_{1,m}\delta_{1,n-1}nb}{4},
\end{eqnarray}
where $\delta_{\mathbb{Z}}(x)$ is as in \eqref{eq1.5}.
\end{theorem}

\begin{proof}
We consider separately the case of $m=n=1$, $m=1$ and $n\geq2$, $m\geq2$ and $n=1$, $m\geq2$ and $n\geq2$. It is easy from \eqref{eq2.15} and \eqref{eq2.22} to check that when $m=n=1$, the right hand side of \eqref{eq3.20} equals
\begin{eqnarray}\label{eq3.21}
&&a\overline{B}_{1}(bx+z)-\mathrm{sgn}(b)\sum_{l=1}^{|b|}\overline{B}_{1}\biggl(\frac{a(l+z)}{b}-y\biggl)\nonumber\\
&&\quad+b\overline{B}_{1}(ax+y)-\mathrm{sgn}(a)\sum_{l=1}^{|a|}\overline{B}_{1}\biggl(\frac{b(l+y)}{a}-z\biggl).
\end{eqnarray}
Notice that from \eqref{eq2.22} and the property of residue systems, we find that for $j\in\mathbb{N}_{0}$, $a,b\in\mathbb{Z}$, $x\in\mathbb{R}$ with $a\not=0$ and $b\not=0$,
\begin{eqnarray}\label{eq3.22}
\sum_{l=1}^{|b|}\overline{B}_{j}\biggl(\frac{al-x}{b}\biggl)&=&(a,b)\sum_{l=1}^{|\frac{b}{(a,b)}|}\overline{B}_{j}\biggl(\frac{\frac{al}{(a,b)}-\frac{x}{(a,b)}}{\frac{b}{(a,b)}}\biggl)\nonumber\\
&=&(a,b)\sum_{l=1}^{|\frac{b}{(a,b)}|}\overline{B}_{j}\biggl(\frac{l-\frac{x}{(a,b)}}{\frac{b}{(a,b)}}\biggl)\nonumber\\
&=&(-1)^{j}b^{1-j}(a,b)^{j}\mathrm{sgn}(b)\overline{B}_{j}\biggl(\frac{x}{(a,b)}\biggl).
\end{eqnarray}
It follows from \eqref{eq2.15} and \eqref{eq3.22} that \eqref{eq3.21} is equal to
\begin{equation*}
a\overline{B}_{1}(bx+z)+b\overline{B}_{1}(ax+y).
\end{equation*}
This says that Theorem \ref{thm3.3} holds true in the case when $m=n=1$.
We now discuss the case $m=1$ and $n\geq2$. Obviously, the case $m=1$ in Theorem \ref{thm3.1} gives that for $n\in\mathbb{N}$, $a,b\in\mathbb{Z}$, $x,y,z\in\mathbb{R}$ with $a\not=0$ and $b\not=0$,
\begin{eqnarray}\label{eq3.23}
&&\overline{B}_{1}(ax+y)\overline{B}_{n}(bx+z)\nonumber\\
&&=\frac{nab^{n-1}\mathrm{sgn}(b)}{n+1}\sum_{l=1}^{|b|}\overline{B}_{n+1}\biggl(x+\frac{l+z}{b}\biggl)\nonumber\\
&&\quad-b^{n-1}\mathrm{sgn}(b)\sum_{l=1}^{|b|}\overline{B}_{1}\biggl(\frac{a(l+z)}{b}-y\biggl)\overline{B}_{n}\biggl(x+\frac{l+z}{b}\biggl)\nonumber\\
&&\quad+\frac{\mathrm{sgn}(a)}{n+1}\sum_{j=0}^{n}\binom{n+1}{j}(-1)^{j}b^{n-j}\sum_{l=1}^{|a|}\overline{B}_{j}\biggl(\frac{b(l+y)}{a}-z\biggl)\overline{B}_{n+1-j}\biggl(x+\frac{l+y}{a}\biggl)\nonumber\\
&&\quad+\frac{(-1)^{n-1}(a,b)^{n+1}\overline{B}_{n+1}\bigl(\frac{by-az}{(a,b)}\bigl)}{a^{n}b(n+1)}\nonumber\\
&&\quad-\delta_{1,n}\mathrm{sgn}(ab)\frac{\delta_{\mathbb{Z}}(ax+y)\delta_{\mathbb{Z}}(bx+z)}{4}.
\end{eqnarray}
Hence, we know from \eqref{eq2.15}, \eqref{eq2.22} and \eqref{eq3.22} that \eqref{eq3.23} can be rewritten as
\begin{eqnarray}\label{eq3.24}
&&\overline{B}_{1}(ax+y)\overline{B}_{n}(bx+z)+\frac{a}{b(n+1)}\overline{B}_{n+1}(bx+z)\nonumber\\
&&=ab^{n-1}\mathrm{sgn}(b)\sum_{l=1}^{|b|}\overline{B}_{n+1}\biggl(x+\frac{l+z}{b}\biggl)\nonumber\\
&&\quad-b^{n-1}\mathrm{sgn}(b)\sum_{l=1}^{|b|}\overline{B}_{1}\biggl(\frac{a(l+z)}{b}-y\biggl)\overline{B}_{n}\biggl(x+\frac{l+z}{b}\biggl)\nonumber\\
&&\quad+\frac{\mathrm{sgn}(a)}{n+1}\sum_{j=0}^{n+1}\binom{n+1}{j}(-1)^{j}b^{n-j}\sum_{l=1}^{|a|}\overline{B}_{j}\biggl(\frac{b(l+y)}{a}-z\biggl)\overline{B}_{n+1-j}\biggl(x+\frac{l+y}{a}\biggl)\nonumber\\
&&\quad-\delta_{1,n}\mathrm{sgn}(ab)\frac{\delta_{\mathbb{Z}}(ax+y)\delta_{\mathbb{Z}}(bx+z)}{4}.
\end{eqnarray}
Multiplying both sides of \eqref{eq3.24} by $b(n+1)$, we prove Theorem \ref{thm3.3} in the case when $m=1$ and $n\geq2$. Similarly, the case $n=1$ in Theorem \ref{thm3.1} means that the case $m\geq2$ and $n=1$ in Theorem \ref{thm3.3} is complete. It remains to consider the case $m\geq2$ and $n\geq2$. Taking the first derivation on both sides of \eqref{eq2.1} with respect to $x$, we conclude that for $n\in\mathbb{N}$, $x\in\mathbb{R}$ with $n\geq2$,
\begin{equation}\label{eq3.25}
\frac{\partial}{\partial x}\bigl(\overline{B}_{n}(x)\bigl)=n\overline{B}_{n-1}(x).
\end{equation}
Therefore, by making the operation $\partial/\partial x$ on both sides of \eqref{eq3.1} and then using \eqref{eq3.25}, we see that Theorem \ref{thm3.3} holds true in the case when $m\geq2$ and $n\geq2$. This proves Theorem \ref{thm3.3}.
\end{proof}

It follows that we state the following reciprocity formula for the sums \eqref{eq3.16}.

\begin{corollary}\label{cor3.4} Let $m,n\in\mathbb{N}$, $a,b\in\mathbb{Z}$ with $a\not=0$ and $b\not=0$. Then, for $x,y,z\in\mathbb{R}$,
\begin{eqnarray}\label{eq3.26}
&&nb^{n-1}\mathrm{sgn}(b)\sum_{j=0}^{m}\binom{m}{j}(-1)^{m-j}a^{m-j}s_{j,m+n-j-1}(a,b;x,y)\nonumber\\
&&\quad-ma^{m-1}\mathrm{sgn}(a)\sum_{j=0}^{n}\binom{n}{j}(-1)^{n-j}b^{n-j}s_{j,m+n-j-1}(b,a;y,x)\nonumber\\
&&=-\mathrm{sgn}(ab)\delta_{\mathbb{Z}}(x)\delta_{\mathbb{Z}}(y)\frac{\delta_{1,m}\delta_{1,n-1}nb-\delta_{1,m-1}\delta_{1,n}ma}{4}\nonumber\\
&&\quad+nb\overline{B}_{m}(x)\overline{B}_{n-1}(y)-ma\overline{B}_{m-1}(x)\overline{B}_{n}(y),
\end{eqnarray}
where $\delta_{\mathbb{Z}}(x)$ is as in \eqref{eq1.5}.
\end{corollary}

\begin{proof}
Taking $x\in\mathbb{Z}$ in Theorem \ref{thm3.3}, we obtain that for $m,n\in\mathbb{N}$, $a,b\in\mathbb{Z}$, $y,z\in\mathbb{R}$ with $a\not=0$ and $b\not=0$,
\begin{eqnarray}\label{eq3.27}
&&nb^{n-1}\mathrm{sgn}(b)\sum_{j=0}^{m}\binom{m}{j}(-1)^{j}a^{m-j}\sum_{l=1}^{|b|}\overline{B}_{j}\biggl(\frac{a(l+z)}{b}-y\biggl)\overline{B}_{m+n-j-1}\biggl(\frac{l+z}{b}\biggl)\nonumber\\
&&\quad+ma^{m-1}\mathrm{sgn}(a)\sum_{j=0}^{n}\binom{n}{j}(-1)^{j}b^{n-j}\sum_{l=1}^{|a|}\overline{B}_{j}\biggl(\frac{b(l+y)}{a}-z\biggl)\overline{B}_{m+n-j-1}\biggl(\frac{l+y}{a}\biggl)\nonumber\\
&&=\mathrm{sgn}(ab)\delta_{\mathbb{Z}}(y)\delta_{\mathbb{Z}}(z)\frac{\delta_{1,m-1}\delta_{1,n}ma+\delta_{1,m}\delta_{1,n-1}nb}{4}\nonumber\\
&&\quad+ma\overline{B}_{m-1}(y)\overline{B}_{n}(z)+nb\overline{B}_{m}(y)\overline{B}_{n-1}(z).
\end{eqnarray}
If we replace $y$ by $-x$, $z$ by $y$ in \eqref{eq3.27}, then we find from \eqref{eq2.15} and the property of residue systems that
\begin{eqnarray}\label{eq3.28}
&&nb^{n-1}\mathrm{sgn}(b)\sum_{j=0}^{m}\binom{m}{j}(-1)^{j}a^{m-j}s_{j,m+n-j-1}(a,b;x,y)\nonumber\\
&&\quad+(-1)^{m+n-1}ma^{m-1}\mathrm{sgn}(a)\sum_{j=0}^{n}\binom{n}{j}(-1)^{j}b^{n-j}s_{j,m+n-j-1}(b,a;y,x)\nonumber\\
&&=\mathrm{sgn}(ab)\delta_{\mathbb{Z}}(x)\delta_{\mathbb{Z}}(y)\frac{\delta_{1,m-1}\delta_{1,n}ma+\delta_{1,m}\delta_{1,n-1}nb}{4}\nonumber\\
&&\quad+(-1)^{m-1}ma\overline{B}_{m-1}(x)\overline{B}_{n}(y)+(-1)^{m}nb\overline{B}_{m}(x)\overline{B}_{n-1}(y).
\end{eqnarray}
Therefore, we get the desired result immediately after multiplying both sides of \eqref{eq3.28} by $(-1)^{m}$.
\end{proof}

We remark that the case $a,b$ being two relatively prime positive integers in Corollary \ref{cor3.4} is in agreement with Carlitz's \cite[Equation (1.14)]{carlitz5} reciprocity theorem. Note that the $n$-th Bernoulli function appearing in the generalized Dedekind-Rademacher sums introduced by Carlitz \cite{carlitz5} is $B_{n}(\{x\})$ rather than $\overline{B}_{n}(x)$.

\section{More general reciprocity formulas}

In this section, we shall use Theorems \ref{thm3.1} and \ref{thm3.3} to give some more general reciprocity formulas for the generalized Dedekind-Rademacher sums \eqref{eq1.13}. We first present the following reciprocity formula.

\begin{theorem}\label{thm4.1} Let $m,n\in\mathbb{N}$, $a,b,c\in\mathbb{Z}$ with $a\not=0,b\not=0$ and $c\not=0$. Then, for $x,y,z\in\mathbb{R}$,
\begin{eqnarray}\label{eq4.1}
&&s_{m,n}\left(
\begin{matrix}
a & b & c \\
x & y & z
\end{matrix}
\right)\nonumber\\
&&=nb^{n-1}\mathrm{sgn}(bc)\sum_{j=0}^{m}\binom{m}{j}\frac{(-1)^{m+n-j}a^{m-j}}{c^{m+n-j-1}(m+n-j)}s_{j,m+n-j}\left(
\begin{matrix}
a & c & b \\
x & z & y
\end{matrix}
\right)\nonumber\\
&&\quad+ma^{m-1}\mathrm{sgn}(ac)\sum_{j=0}^{n}\binom{n}{j}\frac{(-1)^{m+n-j}b^{n-j}}{c^{m+n-j-1}(m+n-j)}s_{j,m+n-j}\left(
\begin{matrix}
b & c & a \\
y & z & x
\end{matrix}
\right)\nonumber\\
&&\quad+\frac{(-1)^{n-1}m!n!c(a,b)^{m+n}\mathrm{sgn}(c)\overline{B}_{m+n}\bigl(\frac{ay-bx}{(a,b)}\bigl)}{a^{n}b^{m}(m+n)!}-\delta_{1,m}\delta_{1,n}\mathrm{sgn}(ab)\frac{\widetilde{N}}{4},
\end{eqnarray}
where $(a,b)$ is as in \eqref{eq1.7}, $\widetilde{N}$ is the number of distinct triples $r,s,t\in\mathbb{Z}$ such that
\begin{equation*}
0\leq\mathrm{sgn}(c)\frac{r+x}{a}=\mathrm{sgn}(c)\frac{s+y}{b}=\mathrm{sgn}(c)\frac{t+z}{c}<1.
\end{equation*}
\end{theorem}

\begin{proof}
Substituting $-(r+x)/c$ for $x$, and then making the operation $\sum_{r=1}^{|c|}$ on both sides of \eqref{eq3.1}, we show from \eqref{eq2.15} and \eqref{eq2.22} that for $m,n\in\mathbb{N}$, $a,b,c\in\mathbb{Z}$, $x,y,z\in\mathbb{R}$ with $a\not=0,b\not=0$ and $c\not=0$,
\begin{eqnarray}\label{eq4.2}
&&(-1)^{m+n}\sum_{r=1}^{|c|}\overline{B}_{m}\biggl(\frac{a(r+x)}{c}-y\biggl)\overline{B}_{n}\biggl(\frac{b(r+x)}{c}-z\biggl)\nonumber\\
&&=nb^{n-1}\mathrm{sgn}(bc)\sum_{j=0}^{m}\binom{m}{j}\frac{(-1)^{j}a^{m-j}}{c^{m+n-j-1}(m+n-j)}\sum_{l=1}^{|b|}\overline{B}_{j}\biggl(\frac{a(l+z)}{b}-y\biggl)\nonumber\\
&&\qquad\times\overline{B}_{m+n-j}\biggl(\frac{c(l+z)}{b}-x\biggl)\nonumber\\
&&\quad+ma^{m-1}\mathrm{sgn}(ac)\sum_{j=0}^{n}\binom{n}{j}\frac{(-1)^{j}b^{n-j}}{c^{m+n-j-1}(m+n-j)}\sum_{l=1}^{|a|}\overline{B}_{j}\biggl(\frac{b(l+y)}{a}-z\biggl)\nonumber\\
&&\qquad\times\overline{B}_{m+n-j}\biggl(\frac{c(l+y)}{a}-x\biggl)\nonumber\\
&&\quad+\frac{(-1)^{m-1}m!n!c(a,b)^{m+n}\mathrm{sgn}(c)\overline{B}_{m+n}\bigl(\frac{az-by}{(a,b)}\bigl)}{a^{n}b^{m}(m+n)!}\nonumber\\
&&\quad-\frac{\delta_{1,m}\delta_{1,n}\mathrm{sgn}(ab)}{4}\sum_{r=1}^{|c|}\delta_{\mathbb{Z}}\biggl(\frac{a(r+x)}{c}-y\biggl)\delta_{\mathbb{Z}}\biggl(\frac{b(r+x)}{c}-z\biggl).
\end{eqnarray}
Since $t+[z]$ runs over a complete residue system modulo $|c|$ as $t$ does, we find that
\begin{eqnarray}\label{eq4.3}
&&\sum_{t=1}^{|c|}\delta_{\mathbb{Z}}\biggl(\frac{a(t+z)}{c}-x\biggl)\delta_{\mathbb{Z}}\biggl(\frac{b(t+z)}{c}-y\biggl)\nonumber\\
&&=\sum_{t=0}^{|c|-1}\delta_{\mathbb{Z}}\biggl(\frac{a(t+\{z\})}{c}-\{x\}\biggl)\delta_{\mathbb{Z}}\biggl(\frac{b(t+\{z\})}{c}-\{y\}\biggl)\nonumber\\
&&=\#\biggl\{t\biggl|\frac{a(t+\{z\})}{c}-\{x\}=r\in\mathbb{Z}, \frac{b(t+\{z\})}{c}-\{y\}=s\in\mathbb{Z},0\leq t\leq |c|-1\biggl\}\nonumber\\
&&=\#\biggl\{t\biggl|\frac{r+\{x\}}{a}=\frac{s+\{y\}}{b}=\frac{t+\{z\}}{c},r,s\in\mathbb{Z},0\leq t\leq |c|-1\biggl\}\nonumber\\
&&=\#\biggl\{t\biggl|0\leq\frac{r+\{x\}}{\mathrm{sgn}(c)a}=\frac{s+\{y\}}{\mathrm{sgn}(c)b}=\frac{t+\{z\}}{\mathrm{sgn}(c)c}<1,r,s,t\in\mathbb{Z}\biggl\},
\end{eqnarray}
where $\#$ denotes the cardinality of a set $S$.
Thus, by replacing $y$ by $x$, $z$ by $y$, $x$ by $z$, and then multiplying both sides of \eqref{eq4.2} by $(-1)^{m+n}$, in view of \eqref{eq4.3}, we get the desired result. This completes the proof of Theorem \ref{thm4.1}.
\end{proof}

It follows that we show some special cases of Theorem \ref{thm4.1}. It is evident from \eqref{eq2.15} and \eqref{eq3.22} that the case $m=n=1$ and $a,b,c\in\mathbb{N}$ in Theorem \ref{thm4.1} gives Berndt's reciprocity formula \eqref{eq1.7}.

For $n\in\mathbb{N}_{0}$, $a,b\in\mathbb{Z}$, $x,y\in\mathbb{R}$ with $b\not=0$, if we set
\begin{equation}\label{eq4.4}
s_{n}(a,b;x,y)=\sum_{r=1}^{|b|}\overline{B}_{1}\biggl(\frac{r+y}{b}\biggl)\overline{B}_{n}\biggl(\frac{a(r+y)}{b}+x\biggl),
\end{equation}
then we deduce from Theorem \ref{thm4.1} the following reciprocity formula.

\begin{corollary}\label{cor4.2} Let $n\in\mathbb{N}_{0}$, $a,b\in\mathbb{N}$. Then, for $x,y\in\mathbb{R}$,
\begin{eqnarray}\label{eq4.5}
&&ab^{n}s_{n}(a,b;x,y)+ba^{n}s_{n}(b,a;y,x)\nonumber\\
&&=-\delta_{1,n}\frac{\delta_{\mathbb{Z}}(x)\delta_{\mathbb{Z}}(y)ab}{4}+\frac{1}{n+1}\sum_{j=0}^{n+1}\binom{n+1}{j}a^{j}b^{n+1-j}\overline{B}_{j}(y)\overline{B}_{n+1-j}(x)\nonumber\\
&&\quad+\frac{n(a,b)^{n+1}\overline{B}_{n+1}\bigl(\frac{ay+bx}{(a,b)}\bigl)}{n+1},
\end{eqnarray}
where $\delta_{\mathbb{Z}}(x)$ is as in \eqref{eq1.5}, $(a,b)$ is as in \eqref{eq1.7}.
\end{corollary}

\begin{proof}
It is easy from \eqref{eq2.21} to see that Corollary \ref{cor4.2} holds true in the case when $n=0$. We next consider the case $n\geq1$. By taking $m=a=1$ and $x\in\mathbb{Z}$ in Theorem \ref{thm4.1}, in light of \eqref{eq2.15} and \eqref{eq3.22}, we obtain that for $n,b,c\in\mathbb{N}$, $y,z\in\mathbb{R}$,
\begin{eqnarray}\label{eq4.6}
&&s_{n}(b,c;-y,z)-(-1)^{n}\frac{b^{n-1}}{c^{n-1}}s_{n}(c,b;-z,y)\nonumber\\
&&=-\delta_{1,n}\frac{\delta_{\mathbb{Z}}(y)\delta_{\mathbb{Z}}(z)}{4}+\frac{1}{n+1}\sum_{j=0}^{n+1}\binom{n+1}{j}\frac{b^{n-j}\overline{B}_{j}(-y)\overline{B}_{n+1-j}(z)}{c^{n-j}}\nonumber\\
&&\quad+\frac{(-1)^{n+1}n(b,c)^{n+1}\overline{B}_{n+1}\bigl(\frac{cy-bz}{(b,c)}\bigl)}{bc^{n}(n+1)}.
\end{eqnarray}
If we multiply both sides of \eqref{eq4.6} by $bc^{n}$ then we have
\begin{eqnarray}\label{eq4.7}
&&bc^{n}s_{n}(b,c;-y,z)-(-1)^{n}cb^{n}s_{n}(c,b;-z,y)\nonumber\\
&&=-\delta_{1,n}\frac{\delta_{\mathbb{Z}}(y)\delta_{\mathbb{Z}}(z)bc}{4}+\frac{1}{n+1}\sum_{j=0}^{n+1}\binom{n+1}{j}b^{n+1-j}c^{j}\overline{B}_{j}(-y)\overline{B}_{n+1-j}(z)\nonumber\\
&&\quad+\frac{(-1)^{n+1}n(b,c)^{n+1}\overline{B}_{n+1}\bigl(\frac{cy-bz}{(b,c)}\bigl)}{n+1}.
\end{eqnarray}
Therefore, by replacing $b$ by $a$, $c$ by $b$, $y$ by $-x$, $z$ by $y$ in \eqref{eq4.7}, we get the desired result after using \eqref{eq2.15} and the property of residue systems.
\end{proof}

In particular, the case $n=1$ in Corollary \ref{cor4.2} gives the reciprocity formula \eqref{eq3.19}, and the case $(a,b)=1$ in Corollary \ref{cor4.2} is in agreement with Carlitz's reciprocity formula \eqref{eq1.12}.

For $m,n\in\mathbb{N}_{0}$, $a,b,c\in\mathbb{Z}$ with $c\not=0$,  if we set
\begin{equation}\label{eq4.8}
s_{m,n}(a,b,c)=\sum_{r=1}^{|c|}\overline{B}_{m}\biggl(\frac{ar}{c}\biggl)\overline{B}_{n}\biggl(\frac{br}{c}\biggl),
\end{equation}
then we can use Theorem \ref{thm4.1} to remove the hypothesis $(a,b)=(b,c)=(a,c)=1$ in Hall and Wilson's \cite[Theorem 1]{hall1} reciprocity theorem for the sums \eqref{eq4.8}, and obtain the following result.

\begin{corollary}\label{cor4.3} Let $p$ be an odd positive integer, and let $a,b,c\in\mathbb{N}$. Then, for $r\in\mathbb{N}_{0}$ with $0\leq r\leq p-1$,
\begin{eqnarray}\label{eq4.9}
&&\sum_{j=1}^{r+1}\binom{p+1}{j}\binom{p-j}{p-1-r}(-1)^{j+1}a^{p}b^{p+1-j}c^{j}s_{j,p+1-j}(b,c,a)\nonumber\\
&&\quad+\sum_{j=1}^{p-r}\binom{p+1}{j}\binom{p-j}{r}(-1)^{j+1}a^{p+1-j}b^{p}c^{j}s_{j,p+1-j}(a,c,b)\nonumber\\
&&\quad+\binom{p+1}{r+1}a^{r+1}b^{p-r}c^{p}s_{p-r,r+1}(a,b,c)\nonumber\\
&&=\bigg(\binom{p}{r}a^{p+1}(b,c)^{p+1}+\binom{p}{r+1}b^{p+1}(a,c)^{p+1}+(-1)^{r}c^{p+1}(a,b)^{p+1}\biggl)B_{p+1}\nonumber\\
&&\quad-\delta_{1,p}\frac{abc\widehat{N}}{2},
\end{eqnarray}
where $(a,b)$ is as in \eqref{eq1.7}, $\widehat{N}$ is the number of distinct triples $r,s,t\in\mathbb{Z}$ such that
\begin{equation*}
0\leq\frac{r}{a}=\frac{s}{b}=\frac{t}{c}<1.
\end{equation*}.
\end{corollary}

\begin{proof}
Taking $a,b,c\in \mathbb{N}$ and $x,y,z\in\mathbb{Z}$ in Theorem \ref{thm4.1}, we see from \eqref{eq3.22} that for $m,n,a,b,c\in\mathbb{N}$,
\begin{eqnarray}\label{eq4.10}
&&s_{m,n}(a,b,c)\nonumber\\
&&=(-1)^{n}nb^{n-1}\sum_{j=1}^{m}\binom{m}{j}\frac{(-1)^{m-j}a^{m-j}}{c^{m+n-j-1}(m+n-j)}s_{j,m+n-j}(a,c,b)\nonumber\\
&&\quad+(-1)^{m}ma^{m-1}\sum_{j=1}^{n}\binom{n}{j}\frac{(-1)^{n-j}b^{n-j}}{c^{m+n-j-1}(m+n-j)}s_{j,m+n-j}(b,c,a)\nonumber\\
&&\quad+\frac{(-1)^{m+n}na^{m}(b,c)^{m+n}B_{m+n}}{b^{m}c^{m+n-1}(m+n)}+\frac{(-1)^{m+n}mb^{n}(a,c)^{m+n}B_{m+n}}{a^{n}c^{m+n-1}(m+n)}\nonumber\\
&&\quad+\frac{(-1)^{n-1}m!n!c(a,b)^{m+n}B_{m+n}}{a^{n}b^{m}(m+n)!}-\delta_{1,m}\delta_{1,n}\frac{\widehat{N}}{4}.
\end{eqnarray}
If we replace $m$ by $p-r$, and $n$ by $r+1$ in \eqref{eq4.10}, then we have
\begin{eqnarray}\label{eq4.11}
&&s_{p-r,r+1}(a,b,c)\nonumber\\
&&=(r+1)b^{r}\sum_{j=1}^{p-r}\binom{p-r}{j}\frac{(-1)^{j}a^{p-r-j}}{c^{p-j}(p+1-j)}s_{j,p+1-j}(a,c,b)\nonumber\\
&&\quad+(p-r)a^{p-r-1}\sum_{j=1}^{r+1}\binom{r+1}{j}\frac{(-1)^{j}b^{r+1-j}}{c^{p-j}(p+1-j)}s_{j,p+1-j}(b,c,a)\nonumber\\
&&\quad+\frac{(r+1)a^{p-r}(b,c)^{p+1}B_{p+1}}{b^{p-r}c^{p}(p+1)}+\frac{(p-r)b^{r+1}(a,c)^{p+1}B_{p+1}}{a^{r+1}c^{p}(p+1)}\nonumber\\
&&\quad+\frac{(-1)^{r}(p-r)!(r+1)!c(a,b)^{p+1}B_{p+1}}{a^{r+1}b^{p-r}(p+1)!}-\delta_{1,p-r}\delta_{1,r+1}\frac{\widehat{N}}{4}.
\end{eqnarray}
Now \eqref{eq4.9} follows when multiplying both sides of \eqref{eq4.11} by $\binom{p+1}{r+1}a^{r+1}b^{p-r}c^{p}$. This completes the proof of Corollary \ref{cor4.3}.
\end{proof}

We next use Theorem \ref{thm3.3} to show another reciprocity formula for the generalized Dedekind-Rademacher sums \eqref{eq1.13} as follows.

\begin{theorem}\label{thm4.4} Let $m,n\in\mathbb{N}$, $a,b,c\in\mathbb{Z}$ with $a\not=0,b\not=0$ and $c\not=0$. Then, for $x,y,z\in\mathbb{R}$,
\begin{eqnarray}\label{eq4.12}
&&mas_{m-1,n}\left(
\begin{matrix}
a & b & c \\
x & y & z
\end{matrix}
\right)+nbs_{m,n-1}\left(
\begin{matrix}
a & b & c \\
x & y & z
\end{matrix}
\right)\nonumber\\
&&=(-1)^{n-1}nb^{n-1}\mathrm{sgn}(bc)\sum_{j=0}^{m}\binom{m}{j}\frac{(-1)^{m-j}a^{m-j}}{c^{m+n-j-2}}s_{j,m+n-j-1}\left(
\begin{matrix}
a & c & b \\
x & z & y
\end{matrix}
\right)\nonumber\\
&&\quad+(-1)^{m-1}ma^{m-1}\mathrm{sgn}(ac)\sum_{j=0}^{n}\binom{n}{j}\frac{(-1)^{n-j}b^{n-j}}{c^{m+n-j-2}}s_{j,m+n-j-1}\left(
\begin{matrix}
b & c & a \\
y & z & x
\end{matrix}
\right)\nonumber\\
&&\quad-\mathrm{sgn}(ab)\frac{\delta_{1,m-1}\delta_{1,n}ma+\delta_{1,m}\delta_{1,n-1}nb}{4}\widetilde{N},
\end{eqnarray}
where $\widetilde{N}$ is as in \eqref{eq4.1}.
\end{theorem}

\begin{proof}
By substituting $-(r+x)/c$ for $x$, and then making the operation $\sum_{r=1}^{|c|}$ on both sides of \eqref{eq3.20}, we know from \eqref{eq2.15} and \eqref{eq2.22} that for $m,n\in\mathbb{N}$, $a,b,c\in\mathbb{Z}$, $x,y,z\in\mathbb{R}$ with $a\not=0,b\not=0$ and $c\not=0$,
\begin{eqnarray}\label{eq4.13}
&&(-1)^{m+n-1}ma\sum_{r=1}^{|c|}\overline{B}_{m-1}\biggl(\frac{a(r+x)}{c}-y\biggl)\overline{B}_{n}\biggl(\frac{b(r+x)}{c}-z\biggl)\nonumber\\
&&\quad+(-1)^{m+n-1}nb\sum_{r=1}^{|c|}\overline{B}_{m}\biggl(\frac{a(r+x)}{c}-y\biggl)\overline{B}_{n-1}\biggl(\frac{b(r+x)}{c}-z\biggl)\nonumber\\
&&=nb^{n-1}\mathrm{sgn}(bc)\sum_{j=0}^{m}\binom{m}{j}\frac{(-1)^{j}a^{m-j}}{c^{m+n-j-2}}\sum_{l=1}^{|b|}\overline{B}_{j}\biggl(\frac{a(l+z)}{b}-y\biggl)\nonumber\\
&&\qquad\times\overline{B}_{m+n-j-1}\biggl(\frac{c(l+z)}{b}-x\biggl)\nonumber\\
&&\quad+ma^{m-1}\mathrm{sgn}(ac)\sum_{j=0}^{n}\binom{n}{j}\frac{(-1)^{j}b^{n-j}}{c^{m+n-j-2}}\sum_{l=1}^{|a|}\overline{B}_{j}\biggl(\frac{b(l+y)}{a}-z\biggl)\nonumber\\
&&\qquad\times\overline{B}_{m+n-j-1}\biggl(\frac{c(l+y)}{a}-x\biggl)\nonumber\\
&&\quad-\mathrm{sgn}(ab)\frac{\delta_{1,m-1}\delta_{1,n}ma+\delta_{1,m}\delta_{1,n-1}nb}{4}\nonumber\\
&&\qquad\times\sum_{r=1}^{|c|}\delta_{\mathbb{Z}}\biggl(\frac{a(r+x)}{c}-y\biggl)\delta_{\mathbb{Z}}\biggl(\frac{b(r+x)}{c}-z\biggl).
\end{eqnarray}
Therefore, replacing $y$ by $x$, $z$ by $y$, $x$ by $z$, and then multiplying both sides of \eqref{eq4.13} by $(-1)^{m+n-1}$, in view of \eqref{eq4.3}, we obtain \eqref{eq4.12} and finish the proof of Theorem \ref{thm4.4}.
\end{proof}

As a special case of Theorem \ref{thm4.4}, we give the following reciprocity formula for the sums \eqref{eq4.8}.

\begin{corollary}\label{cor4.5} Let $m,n,a,b,c\in\mathbb{N}$. Then
\begin{eqnarray}\label{eq4.14}
&&nb^{n-1}c^{m+1}\sum_{j=0}^{m}\binom{m}{j}(-1)^{m-j}a^{m-j}c^{j-m}s_{j,m+n-j-1}(a,c,b)\nonumber\\
&&\quad-ma^{m-1}c^{n+1}\sum_{j=0}^{n}\binom{n}{j}(-1)^{n-j}b^{n-j}c^{j-n}s_{j,m+n-j-1}(b,c,a)\nonumber\\
&&=nbc^{m+n-1}s_{m,n-1}(a,-b,c)-mac^{m+n-1}s_{n,m-1}(b,-a,c)\nonumber\\
&&\quad-\frac{\delta_{1,m}\delta_{1,n-1}nbc^{2}-\delta_{1,m-1}\delta_{1,n}mac^{2}}{4}\widehat{N},
\end{eqnarray}
where $\widehat{N}$ is as in \eqref{eq4.9}.
\end{corollary}

\begin{proof}
Since for $m,n\in\mathbb{N}_{0}$, $a,b,c\in\mathbb{Z}$ with $c\not=0$,
\begin{equation}\label{eq4.15}
s_{m,n}(a,b,c)=(-1)^{m+n}s_{m,n}(a,b,c),
\end{equation}
by taking $a,b,c\in\mathbb{N}$, $x,y,z\in\mathbb{Z}$, and then multiplying both sides of \eqref{eq4.12} by $(-1)^{m}c^{m+n-1}$, in light of \eqref{eq4.15}, we have
\begin{eqnarray*}
&&-mac^{m+n-1}s_{n,m-1}(b,-a,c)+nbc^{m+n-1}s_{m,n-1}(a,-b,c)\nonumber\\
&&=nb^{n-1}c^{m+1}\sum_{j=0}^{m}\binom{m}{j}(-1)^{m-j}a^{m-j}c^{j-m}s_{j,m+n-j-1}(a,c,b)\nonumber\\
&&\quad-ma^{m-1}c^{n+1}\sum_{j=0}^{n}\binom{n}{j}(-1)^{n-j}b^{n-j}c^{j-n}s_{j,m+n-j-1}(b,c,a)\nonumber\\
&&\quad-\frac{\delta_{1,m-1}\delta_{1,n}mac^{2}-\delta_{1,m}\delta_{1,n-1}nbc^{2}}{4}\widehat{N},
\end{eqnarray*}
as was to be shown.
\end{proof}

It should be noted that an equivalent version of the case $a,b,c$ being three relatively prime positive integers in Corollary \ref{cor4.5} can be found in \cite[Theorem 2]{carlitz4}. In addition, by making similar changes of the variables $m,n$ in Corollary \ref{cor4.5} to the ones in the proof Corollary \ref{cor4.3}, the hypothesis $(a,b)=(b,c)=(a,c)=1$ in Mikol\'{a}s's \cite[Equation (5.5)]{mikolas} and Carlitz's \cite[Equation (3.4)]{carlitz3} reciprocity formulas can be removed.

\section*{Acknowledgements}

This work is supported by the Natural Science Foundation of Sichuan Province (Grant No. 2023NSFSC0065).

\end{document}